\input ./style/arxiv-vmsta.cfg
\documentclass[numbers,compress,v1.0.1]{vmsta}
\usepackage{vtexbibtags}
\usepackage{mathrsfs}
\usepackage{mathbh}

\volume{7}
\issue{2}
\pubyear{2020}
\firstpage{157}
\lastpage{190}
\aid{VMSTA155}
\doi{10.15559/20-VMSTA155}
\articletype{research-article}

\startlocaldefs

\newtheorem{theorem}{Theorem}[section]
\newtheorem{proposition}[theorem]{Proposition}
\newtheorem{lemma}[theorem]{Lemma}
\newtheorem{corollary}[theorem]{Corollary}

\theoremstyle{definition}
\newtheorem{definition}{Definition}
\newtheorem{remark}[theorem]{Remark}

\allowdisplaybreaks
\ifdefined\HCode 
\renewenvironment{description}{\begin{enumerate}}{\end{enumerate}}
\def\index#1{}
\else 
\fi
\endlocaldefs

\begin{document}

\begin{frontmatter}
\pretitle{Research Article}

\title{Irregular barrier reflected BDSDEs with general jumps under stochastic Lipschitz and linear growth conditions}

\author[a]{\inits{M.}\fnms{Mohamed}~\snm{Marzougue}\thanksref{cor1}\ead[label=e1]{mohamed.marzougue@edu.uiz.ac.ma}\orcid{0000-0002-0014-1458}}
\thankstext[type=corresp,id=cor1]{Corresponding author.}
\author[b]{\inits{Y.}\fnms{Yaya}~\snm{Sagna}\ead[label=e2]{sagnayaya88@gmail.com}}
\address[a]{Laboratory of Analysis and Applied Mathematics (LAMA), Faculty of sciences Agadir, \institution{Ibn Zohr University},  80000, Agadir, \cny{Morocco}}
\address[b]{LERSTAD, UFR de Sciences Appliqu\'{e}es et de Technologie, \institution{Universit\'{e} Gaston Berger}, BP 234, Saint-Louis, \cny{Senegal}}

\runtitle{Irregular barrier RBDSDEJs under stochastic Lipschitz and linear growth conditions}


\begin{abstract}
In this paper, a solution is given to reflected backward doubly stochastic
differential equations when the barrier is not necessarily right-continuous,
and the noise is driven by  two independent Brownian motions and an independent
Poisson random measure. The existence and uniqueness of the solution is shown,
firstly when the coefficients are stochastic Lipschitz, and secondly by
weakening the conditions on the stochastic growth coefficient.
\end{abstract}
\begin{keywords}
\kwd{Reflected backward doubly stochastic differential equations}
\kwd{irregular barrier}
\kwd{Mertens decomposition}
\kwd{stochastic Lipschitz condition}
\kwd{stochastic linear growth condition}
\end{keywords}
\begin{keywords}[MSC2010]%
\kwd{60H20}
\kwd{60H30}
\end{keywords}

\received{\sday{6} \smonth{3} \syear{2020}}
\revised{\sday{5} \smonth{5} \syear{2020}}
\accepted{\sday{29} \smonth{5} \syear{2020}}
\publishedonline{\sday{10} \smonth{6} \syear{2020}}

\end{frontmatter}
%

\section{Introduction}

Nonlinear backward stochastic differential equations (BSDEs\index{BSDEs} in short) were
first introduced by Pardoux and Peng \cite{Par-Peng90} with the uniform Lipschitz
condition under which they proved the celebrated existence and uniqueness
result. Since then, the theory of BSDEs\index{BSDEs} has\vadjust{\goodbreak} been intensively developed
in the last years. The great interest in this theory comes from its connections
with many other fields of research, such as mathematical finance
\cite{EPeQ:1997,EQ:1997}, stochastic control and stochastic games
\cite{EH:2003} and partial differential equations \cite{Par-Peng92}. After
Pardoux and Peng introduced the theory of \mbox{BSDEs},\index{BSDEs}
they considered \cite{Par-Peng94} a new kind of BSDEs,\index{BSDEs} that is a class of backward
doubly stochastic differential equations (BDSDEs\index{BDSDEs} in short) with two different
directions of stochastic integrals with respect to two independent Brownian
motions. They proved the existence and uniqueness of solutions to BDSDEs\index{BDSDEs}
under uniform Lipschitz conditions on the coefficients.

In the setting of reflected BSDEs (resp. BDSDEs\index{BDSDEs}), an additional nondecreasing
process is added in order to keep the solution above a certain lower-boundary
process, called barrier (or obstacle), and to do this in a minimal fashion.
The reflected BSDEs (RBSDEs\index{RBSDEs} in short) were introduced by El Karoui et al.
\cite{EKPPQ:1997}, 
again under the uniform Lipschitz condition on the coefficients.
The authors of \cite{EKPPQ:1997} proved the existence and uniqueness results in the case of
a Brownian filtration and a continuous barrier. The reflected BDSDEs (RBDSDEs\index{RBDSDEs}
in short) were introduced by Bahlali et al. \cite{bahlali} 
where the authors
studied the case of RBDSDEs\index{RBDSDEs} with continuous coefficients, and 
proved the existence and uniqueness of the solution.

To the best of our knowledge, the paper by Grigorova et al. \cite{GIOOQ:2017} is the
first 
one which studied RBSDEs in the case where the barrier is not necessarily
right-continuous (just right upper semi-continuous). The authors of \cite{GIOOQ:2017} studied
the existence and uniqueness result under the Lipschitz assumption on the
coefficients in a filtration that supports a Brownian motion and an independent
Poisson random measure.\index{independent Poisson random measure} Later, several authors have studied the RBSDEs\index{RBSDEs}
following 
Grigorova et al. \cite{GIOOQ:2017} (see e.g.
\cite{Akdim:2019,BO:2017,BO:2018,KRS:2019,M:2020,ME:2019}). Recently, Berrhazi
et al. \cite{berrhazi:2019} discussed the case of RBDSDE with a right upper
semi-continuous barrier under Lipschitz coefficients.\index{Lipschitz coefficients}

Our aim in this paper is to extend the work on RBDSDEs\index{RBDSDEs} with jumps (RBDSDEJs\index{RBDSDEJs}
in short) to the case of an irregular barrier\index{irregular barrier} (which is assumed to be not necessarily
right-continuous). The specificity of such equations lies in the fact that
the two independent Brownian motions are coupled with an independent Poisson
random measure.\index{independent Poisson random measure} We'll prove the existence and uniqueness of the solution
to such equations under the so-called stochastic Lipschitz coefficients.\index{Lipschitz coefficients}
The interest in this last condition is based on the fact that, unfortunately, in many applications,
the usual Lipschitz conditions cannot be satisfied. For example, the pricing
of the American claim is equivalent to solving the linear RBDSE
%
\begin{equation}
\label{price}
\left \{
\begin{array}[c]{ll}%
\displaystyle -dV_{t}=(r_{t}V_{t}+\theta _{t}Z_{t})dt-Z_{t}dW_{t}+dK_{t},
\quad V_{T}=\xi _{T};
\\
V_{t}\geq \xi _{t},\quad (V_{t}-\xi _{t})dK_{t}=0\quad \text{a.s. }&
\end{array}
\right .
\end{equation}
where $\xi _{t}$ is the amount received from the seller at time $t$,
$r_{t}$ is the interest rate process and $\theta _{t}$ is the risk premium
process. The additional process $K$ is needed for this problem because
there exists no replicating strategy for the option. We have to use a super-replicating
strategy with a consumption process $K$. The minimality condition\index{minimality condition} on
$K$ just states that we only invest money in the portfolio when
$V_{t}>\xi _{t}$. Here both $r_{t}$ and $\theta _{t}$ are not bounded in
general. So, it is not possible to solve the RBSDE \eqref{price} by the
result of El Karoui et al. \cite{EKPPQ:1997}. Thus, in order to study more
general RBSDEs\index{RBSDEs} (resp. RBDSDEs\index{RBDSDEs}), one needs to relax the uniform Lipschitz
conditions on the coefficients. 
To this direction, several attempts have been
done. Among others, we refer to
\cite{Bah-Elo-N'zi-1,Bah-Elo-N'zi-2,EE:2020,GIOQ:2019,ME:2017,ME:2018,ME:2019,ME:2020}
for the case of BSDEs\index{BSDEs}, and \cite{hu,owo,owo18,sow-sagna} 
for BDSDEs.\index{BDSDEs}

In our paper, we use a generalization of the Doob--Meyer decomposition
called the Mertens decomposition.\index{Mertens decomposition} This decomposition is used for strong optional
supermartingales\index{optional supermartingale} which are not necessarily right-continuous. We also use
some tools from the optimal stopping theory, as well as a generalization
of the It\^{o} formula to the case of a strong optional supermartingale\index{optional supermartingale}
called the Gal'chouk--Lenglart formula due to Lenglart \cite{Len:1980}.

The paper is organized as follows. In Section \ref{s1}, we give some notations,
assumptions and main contributions needed in this paper. In Section \ref{s2}, we prove the existence and uniqueness of the solution to RBDSDEJs\index{RBDSDEJs}
with a stochastic Lipschitz coefficients $(f,g)$\index{Lipschitz coefficients} and an irregular barrier
$\xi $,\index{irregular barrier} and we also give a comparison theorem for solutions. Section \ref{s3} is devoted to prove the existence of a minimal solution\index{minimal solution} to RBDSDEJs\index{RBDSDEJs}
under a stochastic growth coefficient $f$.

\section{Definitions and preliminary results}
\label{s1}

Let $0< T < +\infty $ be a non-random horizon time, $\Omega $ be a non-empty
set, ${\mathscr{F}}$ be a $\sigma $-algebra of sets of $\Omega $ and
${\mathbf{P}}$ be a probability measure defined on $\mathscr{F}$. The triple
$(\Omega ,\mathscr{F},{\mathbf{P}})$ defines a probability space which is assumed
to be complete. We assume there are three mutually independent processes:
\begin{itemize}
\item a $d$-dimensional Brownian motion
$(W_{t})_{t\le T}$,
\item a $\ell $-dimensional Brownian motion
$(B_{t})_{t\le T}$,
\item a random Poisson measure $\mu $ on
$E \times \mathbf{R}_{+}$ with compensator
$\nu (dt, de) =\break  \lambda (de) dt$, where the space
$E = \mathbf{R}^{\ell }-\{0\}$ is equipped with its Borel field
${\mathcal{E}}$ such that
$\{\widetilde{\mu }([0, t]\times \mathcal{B})= (\mu - \nu )[0, t]
\times \mathcal{B} \}$ is a martingale for any
$\mathcal{B}\in {\mathcal{E}}$ satisfying
$\lambda (\mathcal{B})<\infty $. $\lambda $ is a $\sigma $-finite measure
on ${\mathcal{E}}$ and satisfies
\begin{equation*}
\int _{E} (1\wedge |e|^{2}) \lambda (de) < \infty .
\end{equation*}
\end{itemize}

We consider the family $({\mathscr{F}}_{t})_{t \le T}$ given by
\begin{equation*}
{\mathscr{F}}_{t} = {\mathscr{F}}^{W}_{t} \vee {\mathscr{F}}^{B}_{t,T}
\vee {\mathscr{F}}^{\mu }_{t},\quad  0\leq t\leq T,
\end{equation*}
where for any process
$(\eta _{t})_{t\le T}, \; {\mathscr{F}}^{\eta }_{s,t} = \sigma \{\eta _{r}-
\eta _{s}, \; s\le r\le t\}\vee {\mathcal{N}}, \; \; {\mathscr{F}}^{\eta }_{t}
= {\mathscr{F}}^{\eta }_{0,t} $. Here $\mathcal{N}$ denotes the class of
${\mathbf{P}}$-null sets of $\mathscr{F}$. Note that the family
$({\mathscr{F}}_{t})_{t \le T}$ does not constitute a classical filtration.

For an integer $k \ge 1$, $ | \;.\; |$ and
$\left \langle .,.\right \rangle $ stand for the Euclidian norm and the
inner product in $\mathbf{R}^{k}$, $\mathcal{T}_{[t,T]}$ denotes the set
of stopping times $\tau $ such that $\tau \in [t,T]$ and
${\mathscr{P}}$ denotes the $\sigma $-algebra of ${\mathscr{F}}_{t}$-predictable
sets of $\Omega \times [0,T]$.

For every ${\mathscr{F}}_{t}$-measurable process $(a_{t})_{t\leq T}$, we
define an increasing process\break  $(A_{t})_{t\leq T}$ by setting
$A_{t} = \int _{0}^{t}a^{2}_{s}ds$.

For every $\beta > 0$, we consider the following sets (where
${\mathbf{E}}$ denotes the mathematical expectation with respect to the probability
measure ${\mathbf{P}}$):
\begin{itemize}
\item $\displaystyle {\mathscr{S}}^{2}(\mathbf{R}^{k})$ and
$\displaystyle {\mathscr{S}}^{2}_{\beta }(\mathbf{R}^{k})$ are the spaces of
${\mathscr{F}}_{t} $-adapted optional processes
$\Psi : \Omega \times\break  [0, T] \longrightarrow \mathbf{R}^{k}$ which satisfy,
respectively,
\begin{align*}
\left \Vert \Psi \right \Vert ^{2}_{{\mathscr{S}}^{2} ( \mathbf{R}^{k})}
= {\mathbf{E}}\left (\operatorname*{ess\,sup}_{\tau \in \mathcal{T}_{[0,T]}}
|\Psi _{\tau }|^{2} \right ) < + \infty
\end{align*}
and
\begin{align*}
\left \Vert \Psi \right \Vert ^{2}_{{\mathscr{S}}^{2}_{\beta }(
\mathbf{R}^{k})} = {\mathbf{E}}\left (\operatorname*{ess\,sup}_{\tau
\in \mathcal{T}_{[0,T]}} e^{\beta A_{\tau }}|\Psi _{\tau }|^{2} \right ) <
+ \infty .
\end{align*}
\item
$\displaystyle \mathscr{M}^{2} (\mathbf{R}^{k\times d}) $,
$\displaystyle \mathscr{M}^{2}_{\beta }(\mathbf{R}^{k\times d})$ and
$\displaystyle \mathscr{M}^{2,a}_{\beta }(\mathbf{R}^{k})$ are the spaces of
${\mathscr{F}}_{t}$-progressively measurable processes
$\Psi : \Omega \times [0, T] \longrightarrow \mathbf{R}^{k\times d}$ (resp.
$\mathbf{R}^{k}$)  which satisfy, respectively,
\begin{gather*}
\left \Vert \Psi \right \Vert ^{2}_{\mathscr{M}^{2} (\mathbf{R}^{k
\times d})} = {\mathbf{E}}\left ( \int _{0}^{T} |\Psi _{t}|^{2} \, dt
\right ) < + \infty ,
\\
\left \Vert \Psi \right \Vert ^{2}_{\mathscr{M}^{2}_{\beta }(
\mathbf{R}^{k\times d})} = {\mathbf{E}}\left ( \int _{0}^{T} e^{\beta A_{t}}|
\Psi _{t}|^{2} \, dt \right ) < + \infty
\end{gather*}
and
\begin{align*}
\left \Vert  \Psi \right \Vert ^{2}_{\mathscr{M}^{2,a}_{\beta }(
\mathbf{R}^{k})} = {\mathbf{E}}\left (\int _{0}^{T} e^{\beta A_{t}}a^{2}_{t}|
\Psi _{t}|^{2} \, dt\right ) < + \infty .
\end{align*}
\item $\mathscr{L}_{\lambda }$ is the space of
${\mathscr{P}} \otimes {\mathcal{E}}$-measurable mappings
$U: E \longrightarrow \mathbf{R}^{k}$ such that
\begin{align*}
\left \Vert  U \right \Vert ^{2}_{\lambda }=\int _{E}|U(e)|^{2}
\lambda (de) < +\infty .
\end{align*}
\item $ \mathscr{L}^{2}_{\beta }(\mathbf{R}^{k})$ is the space
of ${\mathscr{P}} \otimes {\mathcal{E}}$-measurable processes
$ U : \Omega \times [0, T] \times E \longrightarrow \mathbf{R}^{k}$ such
that
\begin{equation*}
\left \Vert U \right \Vert ^{2}_{\mathscr{L}^{2}_{\beta }(\mathbf{R}^{k})}
= {\mathbf{E}}\left (\int _{0}^{T} e^{\beta A_{t}} \left \Vert  U_{t}
\right \Vert ^{2}_{\lambda }dt\right ) < + \infty .
\end{equation*}
\end{itemize}

Notice that the space
\begin{equation*}
\mathcal{A}^{2}_{\beta }(\mathbf{R}^{k})=\mathscr{M}^{2,a}_{\beta }(
\mathbf{R}^{k})\times \mathscr{M}^{2}_{\beta }(\mathbf{R}^{k\times d})
\times \mathscr{L}^{2}_{\beta }(\mathbf{R}^{k})
\end{equation*}
endowed with the norm
\begin{equation*}
\left \Vert (Y, Z, U) \right \Vert ^{2}_{\mathcal{A}^{2}_{\beta }(
\mathbf{R}^{k})}=\left \Vert  Y \right \Vert _{\mathscr{M}^{2,a}_{\beta }(\mathbf{R}^{k})}^{2} + \left \Vert Z \right \Vert _{
\mathscr{M}^{2}_{\beta }(\mathbf{R}^{k\times d})}^{2} + \left \Vert U
\right \Vert _{\mathscr{L}^{2}_{\beta }(\mathbf{R}^{k})}^{2}
\end{equation*}
is a Banach space as is the space
\begin{equation*}
\mathscr{B}^{2}_{\beta }(\mathbf{R}^{k})=(\mathscr{M}^{2,a}_{\beta }(
\mathbf{R}^{k})\cap {\mathscr{S}}^{2}_{\beta }(\mathbf{R}^{k}))
\times \mathscr{M}^{2}_{\beta }(\mathbf{R}^{k\times d})\times
\mathscr{L}^{2}_{\beta }(\mathbf{R}^{k})
\end{equation*}
with the norm
\begin{equation*}
\left \Vert (Y, Z, U) \right \Vert ^{2}_{\mathscr{B}^{2}_{\beta }(
\mathbf{R}^{k})} =\left \Vert  Y \right \Vert _{{\mathscr{S}}^{2}_{\beta }(\mathbf{R}^{k})}^{2} + \left \Vert  Y \right \Vert _{
\mathscr{M}^{2,a}_{\beta }(\mathbf{R}^{k})}^{2} + \left \Vert Z
\right \Vert _{\mathscr{M}^{2}_{\beta }(\mathbf{R}^{k\times d})}^{2} +
\left \Vert U \right \Vert _{\mathscr{L}^{2}_{\beta }( \mathbf{R}^{k})
}^{2} .\vadjust{\goodbreak}
\end{equation*}

For a l\`{a}dl\`{a}g (limited from right and left) process
$(Y_{t})_{t\leq T}$, we denote by:
\begin{itemize}
\item $Y_{t-}=\lim \limits _{s\nearrow t}Y_{s}$ the left-hand limit of
$Y$ at $t\in [0,T]$, $(Y_{0-}=Y_{0})$, $Y_{-}:=(Y_{t-})_{t\leq T}$ and
$\Delta Y_{t}:=Y_{t}-Y_{t-}$ the size of the left jump of $Y$ at $t$.
\item $Y_{t+}=\lim \limits _{s\searrow t}Y_{s}$ the right-hand limit of
$Y$ at $t\in [0,T]$, $(Y_{T+}=Y_{T})$, $Y_{+}:=(Y_{t+})_{t\leq T}$ and
$\Delta _{+}Y_{t}:=Y_{t+}-Y_{t}$ the size of the right jump of $Y$ at
$t$.
\end{itemize}

Let
$f : \Omega \times [0, T]\times \mathbf{R}^{k}\times \mathbf{R}^{k
\times d}\times \mathbf{R}^{k} \to \mathbf{R}^{k}$,
$\, g : \Omega \times [0, T] \times \mathbf{R}^{k}\times \mathbf{R}^{k
\times d}\times \mathbf{R}^{k} \to \mathbf{R}^{k\times \ell } $, and
$(\xi _{t})_{t\leq T}$ be an optional process which is assumed to be right
upper semi-continuous and limited from left. The process $(\xi _{t})_{t\leq T}$ will be
called \textbf{irregular barrier}. We are interested in the following RBDSDEJs
associated with parameters $(f,g,\xi )$:
%
\begin{equation}
\label{backw}
\left \{
\begin{array}{ll}
Y_{\tau }=\xi _{T} +\displaystyle \int _{\tau }^{T} f(s,\Theta _{s})ds+
\int _{\tau }^{T} g(s,\Theta _{s})dB_{s}-\int _{\tau }^{T}Z_{s}dW_{s}&
\text{}
\\
\hspace{0.5cm}
-\displaystyle \int _{\tau }^{T}\int _{E}U_{s}(e)\widetilde{\mu }(ds,de)+K_{T}-K_{
\tau }+C_{T-}-C_{\tau -}\quad \tau \in \mathcal{T}_{[0,T]},& \text{}
\\
Y_{\tau }\geq \xi _{\tau }\quad \forall \tau \in \mathcal{T}_{[0,T]},&
\text{}
\\
K=K^{c}+K^{d} \text{ (continuous + purely discontinuous part) is a}&
\text{}
\\
\text{nondecreasing right-continuous predictable process with}& \text{}
\\
K_{0}=0 \text{ such that }& \text{}
\\
\displaystyle \int _{0}^{T}\mathbh{1}_{\{Y_{t}>\xi _{t}\}}dK^{c}_{t}=0
\ \text{a.s. and } (Y_{\tau -}-\xi _{\tau -})\Delta K^{d}_{\tau }=0 \ \text{a.s.}
\quad \forall \tau \in \mathcal{T}^p_{[0,T]},& \text{}
\\
C \text{ is a nondecreasing right-continuous predictable purely}&
\text{}
\\
\text{discontinuous process with }C_{0-}=0 \text{ such that }& \text{}
\\
(Y_{\tau }-\xi _{\tau })\Delta C_{\tau }=0 \ \text{a.s.} \quad \forall \tau
\in \mathcal{T}_{[0,T]}. & \text{}
\end{array}
\right .
\end{equation}
Here $\Theta _{s}$ stands for the triple $(Y_{s}, Z_{s}, U_{s})$.

Let us consider the filtration $(\mathcal{G}_{t})_{t\leq T}$ given by
$\mathcal{G}_{t} = \mathscr{F}_{t}^{W}\vee \mathscr{F}_{T}^{B}\vee
\mathscr{F}_{t}^{\mu },\,\,\, 0 \leq t\leq T$ which is assumed to be right-continuous
and quasi-left-continuous, and 
make precise the notion of solution
to RBDSDEJ \eqref{backw}.
%
\begin{definition}
Let $\xi $ be an irregular barrier.\index{irregular barrier} A process $(Y,Z,U,K,C)$ is called a
solution to RBDSDEJ associated with parameters $(f,g,\xi )$, if it satisfies
the system \eqref{backw} and
\begin{itemize}
\item $(Y,Z,U)\in \mathscr{B}^{2}_{\beta }(\mathbf{R}^{k})$,
\item
$(K,C)\in \mathscr{S}^{2}(\mathbf{R}^{k})\times \mathscr{S}^{2}(
\mathbf{R}^{k})$.
\end{itemize}
\end{definition}

\begin{remark}
We note that a process
$(Y,Z,U,K,C)\in \mathscr{B}^{2}_{\beta }(\mathbf{R}^{k})\times
\mathscr{S}^{2}(\mathbf{R}^{k})\times \mathscr{S}^{2}(\mathbf{R}^{k})$
satisfies the equation \eqref{backw} if and only if
\begin{eqnarray*}
Y_{t}&=&\xi _{T} +\displaystyle \int _{t}^{T} f(s,\Theta _{s})ds+
\int _{t}^{T} g(s,\Theta _{s})dB_{s}-\int _{t}^{T}Z_{s}dW_{s}
\\
&&-\int _{t}^{T}\int _{E}U_{s}(e)\widetilde{\mu }(ds,de)+K_{T}-K_{t}+C_{T-}-C_{t-}.
\end{eqnarray*}
\end{remark}

\begin{remark}
If $(Y,Z,U,K,C)$ is a solution to RBDSDEJ \eqref{backw}, then
$\Delta C_{t}=Y_{t}- Y_{t+}$ for all $t\leq T$ outside an evanescent set.
It follows that $Y_{t}\geq Y_{t+}$ for all $t\leq T$, which implies that
$Y$ is necessarily right upper semi-continuous. Moreover, the process
$\left (Y_{t}+\int _{0}^{t}f(s,\Theta _{s})ds\right )_{t\leq T}$ is a strong
supermartingale. Actually, by using H\"{o}lder's inequality and the stochastic
Lipschitz condition on $f$ (below), we obtain, for each
$\tau \in \mathcal{T}_{[0,T]}$,
\begin{eqnarray*}
&&\mathbb{E}\left |Y_{\tau }+\int _{0}^{\tau }f(s,\Theta _{s})ds\right |^{2}
\\
&\leq &2\left (\mathbb{E}\left |Y_{\tau }\right |^{2}+\frac{1}{\beta }
\mathbb{E}\int _{0}^{T} e^{\beta A_{s}}\left |
\frac{f(s,\Theta _{s})}{a_{s}}\right |^{2}ds\right )
\\
&\leq & 2\left (\|Y\|_{{\mathscr{S}}^{2}_{\beta }(\mathbf{R}^{k})}^{2}+
\frac{4}{\beta }\|Y\|_{\mathscr{M}^{2,a}_{\beta }(\mathbf{R}^{k})}^{2}+
\frac{4}{\beta }\|Z\|_{\mathscr{M}^{2}_{\beta }(\mathbf{R}^{k\times d})}^{2}+
\frac{4}{\beta }\|U\|_{\mathscr{L}^{2}_{\beta }(\mathbf{R}^{k})}^{2}
\right .
\\
&&\left .+\frac{4}{\beta }\left \|  \frac{f(.,0)}{a}\right \|  _{
\mathscr{M}^{2}_{\beta }(\mathbf{R}^{k})}^{2}\right )<+\infty .
\end{eqnarray*}
Moreover, for all $\tau ,\nu \in \mathcal{T}_{[0,T]}$ with
$\nu \leq \tau $ we have
\begin{equation*}
\mathbb{E}\left [Y_{\tau }-Y_{\nu }-\int _{\nu }^{\tau }f(s,\Theta _{s})ds|
\mathcal{G}_{\nu }\right ]=\mathbb{E}\left [K_{\nu }-K_{\tau }+C_{\nu ^{-}}-C_{
\tau ^{-}}|\mathcal{G}_{\nu }\right ]\;\; a.s.
\end{equation*}
Since $K$ and $C$ are nondecreasing processes, and
$ \left( Y_{t}+\int _{0}^{t}f(s,\Theta _{s})ds \right)_{t\leq T}$ is a
$\mathcal{G}_{t}$-adapted process then the observation follows.
\end{remark}

\begin{remark}%
\label{rem02}
In our framework the filtration is quasi-left-continuous, martingales have
only totally inaccessible jumps\index{totally inaccessible jumps} and $Y$ has two type of left-jumps: totally
inaccessible jumps\index{totally inaccessible jumps} which stem from stochastic integral with respect to
$\widetilde{\mu }$, and predictable jumps\index{predictable jumps} which come from the predictable
jumps\index{predictable jumps} of the irregular barrier\index{irregular barrier} $\xi $. The latter are the source of the
predictability of $K$. Moreover, the processes $K$ and $\mu $ do not have
jumps in common.
\end{remark}

\begin{remark}[The particular case of a right-continuous barrier]
If the barrier $\xi $ is right-continuous, we have
$Y_{t}\geq Y_{t+}\geq \xi _{t+}=\xi _{t}$. Hence, if $t$ is such that
$Y_{t}=\xi _{t}$, then $Y_{t}=Y_{t+}=\xi _{t}$. If $t$ is such that
$Y_{t}>\xi _{t}$, then by the minimality condition\index{minimality condition} on $C$,
$Y_{t}-Y_{t+}=C_{t}-C_{t-}=0$. Thus, in both cases, $Y_{t}=Y_{t+}$, so
$Y$ is right-continuous. Moreover, the right-continuity of $Y$ combined
with the fact that $\Delta C_{t}=Y_{t}- Y_{t+}$ give $C_{t}=C_{t-}$ for
all $t\leq T$. As $C$ is right-continuous, purely discontinuous and such
that $C_{0-}=0$, we deduce $C=0$. Thus, we recover the usual formulation
of RBDSDEJs\index{RBDSDEJs} with a right-continuous barrier.
\end{remark}

\begin{proposition}%
\label{prop}
Let $(Y,Z,U)\in \mathscr{B}^{2}_{\beta }(\mathbf{R}^{k})$ with $Y$ 
being a l\`{a}dl\`{a}g process, and let a coefficient
$g(.)\in \mathscr{M}^{2}_{\beta }(\mathbf{R}^{k\times \ell })$. Then
\begin{equation*}
\left (\displaystyle \int _{0}^{t} e^{\beta A_{s}}\left (\langle Y_{s-},Z_{s}dW_{s}
\rangle + \langle Y_{s-},g(s)dB_{s}\rangle + \int _{E} \langle Y_{s-},U_{s}(e)
\widetilde{\mu }(ds,de)\rangle \right )\right )_{t\leq T}
\end{equation*}
is a martingale.
\end{proposition}
\begin{proof}
Using the left-continuity of trajectories of the process $Y_{s-}$, we
have
\begin{equation*}
|Y_{s-}(\omega )|^{2}\leq \sup _{t\in [0,T]\cap \mathbb{Q}}|Y_{t-}(
\omega )|^{2} \quad \forall (s,\omega )\in [0,T]\times \Omega .
\end{equation*}
On the other hand, we have
$|Y_{t-}|^{2}\leq \operatorname*{ess
\,sup}\limits _{\tau \in \mathcal{T}_{[0,T]}}|Y_{\tau }|^{2}$ which implies
\begin{equation*}
\sup \limits _{t\in [0,T]\cap \mathbb{Q}}|Y_{t-}|^{2}\leq
\operatorname*{ess
\,sup}\limits _{\tau \in \mathcal{T}_{[0,T]}}|Y_{\tau }|^{2}.
\end{equation*}
Then for all $\tau \leq T$
\begin{eqnarray*}
\int _{0}^{\tau }e^{2\beta A_{s}}|Y_{s-}|^{2}|Z_{s}|^{2}ds&\leq &\int _{0}^{\tau }e^{2\beta A_{s}}\sup _{t\in [0,T]\cap \mathbb{Q}}|Y_{t-}|^{2}|Z_{s}|^{2}ds
\\
&\leq &\int _{0}^{\tau }e^{2\beta A_{s}} \operatorname*{ess
\,sup}_{\tau \in \mathcal{T}_{[0,T]}}|Y_{\tau }|^{2}|Z_{s}|^{2}ds.
\end{eqnarray*}
Further, we have
\begin{equation*}
\int _{0}^{\tau }e^{2\beta A_{s}} \operatorname*{ess
\,sup}_{\tau \in \mathcal{T}_{[0,T]}}|Y_{\tau }|^{2}|Z_{s}|^{2}ds
\leq \operatorname*{ess
\,sup}_{\tau \in \mathcal{T}_{[0,T]}}e^{\beta A_{\tau }}|Y_{\tau }|^{2}
\int _{0}^{\tau }e^{\beta A_{s}}|Z_{s}|^{2}ds.
\end{equation*}
Hence
\begin{eqnarray*}
{\mathbf{E}}\sqrt{\int _{0}^{\tau }e^{2\beta A_{s}}|Y_{s-}|^{2}|Z_{s}|^{2}ds}&
\leq & {\mathbf{E}}\sqrt{\operatorname*{ess
\,sup}_{\tau \in \mathcal{T}_{[0,T]}}e^{\beta A_{\tau }}|Y_{\tau }|^{2}
\int _{0}^{T}e^{\beta A_{s}}|Z_{s}|^{2}ds}
\\
&\leq &\frac{1}{2}\left (\|Y\|^{2}_{\mathscr{S}^{2}_{\beta }(
\mathbf{R}^{k})}+\|Z\|^{2}_{\mathscr{M}^{2}_{\beta }(\mathbf{R}^{k
\times d})}\right ).
\end{eqnarray*}
Since
$(Y,Z)\in \mathscr{S}^{2}_{\beta }(\mathbf{R}^{k})\times \mathscr{M}^{2}_{\beta }(\mathbf{R}^{k\times d})$, 
we get the finite expectation. Since the process
$\left (\int _{0}^{t} e^{\beta A_{s}}\langle Y_{s},Z_{s}dW_{s}
\rangle \right )_{t\leq T}$ is adapted, 
it is a martingale.

By the same arguments,
\begin{equation*}
\left (\displaystyle \int _{0}^{t}\int _{E} e^{\beta A_{s}}\langle Y_{s-},U_{s}(e)
\widetilde{\mu }(ds,de)\rangle \right )_{t\leq T}\quad\ \ \text{and}\quad\ \ \left (
\displaystyle \int _{0}^{t}e^{\beta A_{s}}\langle Y_{s-},g(s)dB_{s}
\rangle \right )_{t\leq T}
\end{equation*}
are martingales since
\begin{eqnarray*}
&&%
\hspace{-1.5cm}%
{\mathbf{E}}\sqrt{\int _{0}^{\tau }\int _{E}e^{2\beta A_{s}}|Y_{s-}|^{2}|U_{s}(e)|^{2}
\lambda (de)ds}
\\
&&%
\hspace{3cm}%
\leq \; {\mathbf{E}}\sqrt{\operatorname*{ess\,sup}_{\tau \in \mathcal{T}_{[0,T]}}e^{
\beta A_{\tau }}|Y_{\tau }|^{2}\int _{0}^{T}e^{\beta A_{s}}\|U_{s}\|_{\lambda }^{2}ds}
\\
&&%
\hspace{3cm}%
\leq \;\frac{1}{2}\left (\|Y\|^{2}_{\mathscr{S}^{2}_{\beta }(
\mathbf{R}^{k})}+\|U\|^{2}_{\mathscr{L}^{2}_{\beta }(\mathbf{R}^{k})}
\right )
\end{eqnarray*}
and
\begin{align*}
{\mathbf{E}}\sqrt{\int _{0}^{\tau }e^{2\beta A_{s}}|Y_{s-}|^{2}|g(s)|^{2}ds}&
\leq  {\mathbf{E}}\sqrt{\operatorname*{ess
\,sup}_{\tau \in \mathcal{T}_{[0,T]}}e^{\beta A_{\tau }}|Y_{\tau }|^{2}
\int _{0}^{T}e^{\beta A_{s}}|g(s)|^{2}ds}
\\
&\leq  \frac{1}{2}\left (\|Y\|_{\mathscr{S}^{2}_{\beta }(\mathbf{R}^{k})}^{2}+
\|g\|_{\mathscr{M}^{2}_{\beta }(\mathbf{R}^{k\times \ell })}^{2}
\right ).\qedhere
\end{align*}
\end{proof}

Let us recall some results 
from the general theory of optional processes,
which will be useful in the sequel.

\begin{theorem}[Mertens decomposition\index{Mertens decomposition}]%
\label{ap1}
Let $\tilde{Y}$ be a strong optional supermartingale\index{optional supermartingale} of class (D). There
exists a unique uniformly integrable martingale (c\`{a}dl\`{a}g) $N$, a
unique nondecreasing right-continuous predictable process $K$ with
$K_{0}=0$ and ${\mathbf{E}}|K_{T}|^{2}<+\infty $, and a unique nondecreasing
right-continuous adapted purely discontinuous process $C$ with
$C_{0-}=0$ and ${\mathbf{E}}|C_{T}|^{2}<+\infty $, such that
\begin{equation*}
\tilde{Y}_{t}=N_{t}-K_{t}-C_{t-}\quad \forall t\leq T \; a.s.
\end{equation*}
\end{theorem}
%
\begin{theorem}[Dellacherie--Meyer]
Let $K$ be a nondecreasing predictable process. Let $U$ be the potential
of the process $K$, i.e. $U:={\mathbf{E}}[K_{T}|\mathcal{G}_{t}]-K_{t}$ for
all $t\leq T$. We assume that there exists a positive
$\mathcal{G}_{T}$-measurable random variable $X$ such that
$|U_{\nu }|\leq {\mathbf{E}}[X|\mathcal{G}_{\nu }]$ a.s. for all
$\nu \in \mathcal{T}_{[0,T]}$. Then
${\mathbf{E}}|K_{T}|^{2}\leq c{\mathbf{E}}|X|^{2}$, where $c$ is a positive constant.
\end{theorem}
The proof is established in chapter VI, Theorem 99, \cite{DM} for the case
of a nondecreasing process which is not necessarily right-continuous nor
left-continuous.
%
\begin{corollary}%
\label{ap2}
Let $Y$ be a strong optional supermartingale\index{optional supermartingale} of class (D) such that, for
all $\nu \in \mathcal{T}_{[0,T]}$,
$|Y_{\nu }|\leq {\mathbf{E}}[X|\mathcal{G}_{\nu }]$ a.s., where $X$ is a nonnegative
$\mathcal{G}_{T}$-measurable random variable. Let $\tilde{K}$ be the Mertens
process\index{Mertens process} associated with $Y$. Then there exists a positive constant $c$ such
that ${\mathbf{E}}|\tilde{K}_{T}|^{2}\leq c{\mathbf{E}}|X|^{2}$.
\end{corollary}
The proof is established in \cite{ME:2019}.
%
\begin{theorem}[Gal'chouk--Lenglart formula]%
\label{gal}
Let $n\in \mathbb{N}$. Let $Y$ be an $n$-dimensional optional semimartingale
with the decomposition $Y^{k}=Y^{k}_{0}+M^{k}+R^{k}+O^{k}$, for all
$k=1,\ldots ,n$, where $M^{k}$ is a (c\`{a}dl\`{a}g) local martingale,
$R^{k}$ is a right-continuous process of finite variation such that
$R^{k}_{0} =0$ and $O^{k}$ is a left-continuous process of finite variation
which is purely discontinuous and such that $O^{k}_{0} =0$. Let $F$ be
a twice continuously differentiable function on $\mathbb{R}^{n}$. Then,
almost surely, for all $t\geq 0$,
\begin{eqnarray*}
F(Y_{t})&=&F(Y_{0}) +\sum _{k=1}^{n}\int _{0}^{t}D^{k} F(Y_{s-})d(M^{k}+R^{k})_{s}+
\sum _{k=1}^{n}\int _{0}^{t}D^{k} F(Y_{s})dO^{k}_{s+}
\\
&&+\frac{1}{2}\sum _{k,l=1}^{n}\int _{0}^{t}D^{k}D^{l} F(Y_{s-})d[M^{k,c},M^{l,c}]_{s}
\\
&&+\sum _{0<s\leq t}\left [F(Y_{s})-F(Y_{s-})-\sum _{k=1}^{n}D^{k} F(Y_{s-})
\Delta Y^{k}_{s}\right ]
\\
&&+\sum _{0\leq s<t}\left [F(Y_{s+})-F(Y_{s})-\sum _{k=1}^{n}D^{k} F(Y_{s})
\Delta _{+}Y^{k}_{s}\right ],
\end{eqnarray*}
where $D^{k}$ denotes the differentiation operator with respect to the
$k$-th coordinate, and $M^{k,c}$ denotes the continuous part of
$M^{k}$.
\end{theorem}
%
\begin{corollary}
Let $Y$ be an optional semimartingale\index{optional semimartingale} with the decomposition
$Y=Y_{0}+M+R+O$ where $M$, $R$ and $O$ are as in Theorem \ref{gal}. Then, almost surely, for all $t\geq 0$,
\begin{eqnarray*}
&& e^{\beta A_{t}}|Y_{t}|^{2}+\beta \int _{t}^{T}e^{\beta A_{s}}a^{2}_{s}|Y_{s}|^{2}ds
\\
&=&e^{\beta A_{T}}|Y_{T}|^{2}+ 2\int _{t}^{T}e^{\beta A_{s}} Y_{s-}d(M+R)_{s}+
\int _{t}^{T}e^{\beta A_{s}} Y_{s}dO_{s+}
\\
&&+\int _{t}^{T}e^{\beta A_{s}}d[M^{c},M^{c}]_{s}-\sum _{t<s\leq T}e^{
\beta A_{s}}\left (\Delta Y_{s}\right )^{2}-\sum _{t\leq s<T}e^{
\beta A_{s}}\left (\Delta _{+}Y_{s}\right )^{2}.
\end{eqnarray*}
\end{corollary}
\begin{proof}
To prove the corollary, it suffices to apply the change of variables formula
from Theorem \ref{gal} with $F(X,Y)=XY^{2}$ for
$X_{t}=e^{\beta A_{t}}$.
\end{proof}

\begin{lemma}%
\label{lem1}
Let
$ Y\in {\mathscr{S}}^{2}_{\beta }(\mathbf{R}^{k}),\vartheta \in
\mathscr{M}^{2}_{\beta }(\mathbf{R}^{k}),\zeta \in \mathscr{M}^{2}_{
\beta }(\mathbf{R}^{k\times \ell }),\pi \in \mathscr{M}^{2}_{\beta } (
\mathbf{R}^{k\times d})$ and
$ \phi \in \mathscr{L}^{2}_{\beta }(\mathbf{R}^{k})$ be such that
\begin{equation*}
Y_{t} = Y_{0} - \int _{0}^{t}\vartheta _{s}ds - \int _{0}^{t}\zeta _{s}dB_{s}
+ \int _{0}^{t}\pi _{s} dW_{s} + \int _{0}^{t}\int _{E}\phi _{s}(e)
\widetilde{\mu }(ds,de)-K_{t}-C_{t-},
\end{equation*}
where ${\mathbf{E}}|K_{T}|^{2}+{\mathbf{E}}|C_{T}|^{2}<+\infty $. Then $Y$ is an
optional semimartingale\index{optional semimartingale} with the decomposition $Y=Y_{0}+M+R+O$ where
$M_{t}=- \int _{0}^{t}\zeta _{s}dB_{s} + \int _{0}^{t}\pi _{s} dW_{s} +
\int _{0}^{t}\int _{E}\phi _{s}(e)\widetilde{\mu }(ds,de)$,
$R_{t}= - \int _{0}^{t}\vartheta _{s}ds-K_{t}$ and $O_{t}=-C_{t-}$, and
we have, for any $\beta >0$ and $t\leq T$,
\begin{eqnarray*}
&& e^{\beta A_{t}}|Y_{t}|^{2}+\beta \int _{t}^{T}e^{\beta A_{s}}a^{2}_{s}|Y_{s}|^{2}ds+
\int _{t}^{T}e^{\beta A_{s}}|\pi _{s}|^{2}ds
\\
&=&e^{\beta A_{T}}|Y_{T}|^{2}+ 2\int _{t}^{T}e^{\beta A_{s}}\langle Y_{s-},
\vartheta _{s}\rangle ds+ 2\int _{t}^{T}e^{\beta A_{s}}\langle Y_{s-},
\zeta _{s}dB_{s}\rangle
\\
&& -2\int _{t}^{T}e^{\beta A_{s}}\langle Y_{s-},\pi _{s}dW_{s}
\rangle - 2\int _{t}^{T}\int _{E} e^{\beta A_{s}}\langle Y_{s-},
\phi _{s}(e)\widetilde{\mu }(de ,ds)\rangle
\\
&&+\int _{t}^{T}e^{\beta A_{s}}|\zeta _{s}|^{2}ds+ 2\int _{t}^{T}e^{
\beta A_{s}}\langle Y_{s-},dK_{s}\rangle + 2\int _{t}^{T}e^{\beta A_{s}}
\langle Y_{s},dC_{s}\rangle
\\
&&-\sum _{t<s\leq T}e^{\beta A_{s}}\left (\Delta Y_{s}\right )^{2}-
\sum _{t\leq s<T}e^{\beta A_{s}}\left (\Delta _{+}Y_{s}\right )^{2}.
\end{eqnarray*}
\end{lemma}

\section{Reflected BDSDEJs with stochastic Lipschitz coefficients\index{Lipschitz coefficients}}
\label{s2}

\subsection{Assumptions}

We assume that the parameters $(f,g,\xi )$ satisfy the following assumptions
\textbf{(A1)}, for some $\beta > 0$ (where we define for all
$t\leq T,\,\, h(t, 0) = h(t, 0, 0, 0), \,\,\text{for}\,\, h\in \left
\{  f, g\right \}  $ to ease the reading).
\begin{description}
\item[\bf(A1.1):] $f$ and $g$ are jointly measurable, and
there exists a constant $\alpha \in\ ]0,1[$ and four non-negative,
${\mathscr{F}}^{W}_{t}$-measurable processes
$(\gamma _{t})_{t\le T}$, $(\kappa _{t})_{t\le T}$,
$(\sigma _{t})_{t\le T}$ and $(\varrho _{t})_{t\le T}$ such that for all
$(y, y^{\prime }) \in (\mathbf{R}^{k})^{2}$,
$(z, z^{\prime }) \in (\mathbf{R}^{k\times d})^{2}$ and
$(u,u^{\prime })\in (\mathscr{L}_{\lambda })^{2}$,
\begin{eqnarray*}
&\vert f(t, y, z, u) - f(t, y^{\prime }, z^{\prime }, u^{\prime }) \vert
\le \gamma _{t} |y - y^{\prime }| + \kappa _{t} |z-z^{\prime }| + \sigma _{t}
\left \Vert  u-u^{\prime }\right \Vert _{\lambda },
\\
&\vert g(t, y, z, u) - g(t, y^{\prime }, z^{\prime }, u^{\prime }) \vert ^{2}
\le \varrho _{t}|y - y^{\prime }|^{2} + \alpha \left (|z-z^{\prime }|^{2} +
\left \Vert  u-u^{\prime }\right \Vert ^{2}_{\lambda }\right ).
\end{eqnarray*}
\item[\bf(A1.2):] For all
$0\leq t \leq T,\,\, a^{2}_{t} = \gamma _{t} + \kappa ^{2}_{t} +
\sigma ^{2}_{t} + \varrho _{t} > 0$.
\item[\bf(A1.3):] For any
$(t, y, z, u)\in [0, T]\times \mathbf{R}^{k}\times \mathbf{R}^{k
\times d}\times \mathscr{L}_{\lambda }, f(t,y,z,u)$ and
$g(t,y,z,u)$ are $\mathscr{F}_{t}$-measurable with
$\frac{f(.,0)}{a}\in \mathscr{M}^{2}_{\beta }(\mathbf{R}^{k})$ and
$g(.,0)\in \mathscr{M}^{2}_{\beta }(\mathbf{R}^{k\times \ell })$.
\item[\bf(A1.4):] The irregular barrier\index{irregular barrier}
$(\xi _{t})_{t\le T}$ is in
$\mathscr{S}^{2}_{2\beta }(\mathbf{R}^{k})$.
\end{description}

\subsection{Existence and uniqueness of solution}

Before proving the existence and uniqueness, let us establish the corresponding
result in the case where the coefficients $f$ and $g$ do not depend on
the variables $Y$, $Z$ and $U$. So we consider the RBDSDEJ,
$\forall \tau \in \mathcal{T}_{[0,T]}$,
%
\begin{equation}
\label{eq1}
\left \{
\begin{array}{@{}l}
Y_{\tau }=\xi _{T} +\displaystyle \int _{\tau }^{T} f(s)ds+\int _{\tau }^{T} g(s)dB_{s}-\int _{\tau }^{T}Z_{s}dW_{s}-\int _{\tau }^{T}
\int _{E}U_{s}(e)\widetilde{\mu }(ds,de)
\\
\hspace{1cm}
+K_{T}-K_{\tau }+C_{T-}-C_{\tau -},
\\
Y_{\tau }\geq \xi _{\tau },
\\
\displaystyle \int _{0}^{T}\mathbh{1}_{\{Y_{t}>\xi _{t}\}}dK^{c}_{t}=0,\quad
 (Y_{\tau -}-\xi _{\tau -})\Delta K^{d}_{\tau }=0\quad  \text{and}\quad (Y_{\tau }-
\xi _{\tau })\Delta C_{\tau }=0 \ \text{a.s.}
\end{array}
\right .
\end{equation}
where $K=K^{c}+K^{d}$ (continuous + purely discontinuous part) is a nondecreasing
right-continuous predictable process with $K_{0}=0$ and $C$ is a nondecreasing
right-continuous predictable purely discontinuous process with
$C_{0-}=0$. Moreover, the irregular barrier\index{irregular barrier} $\xi $ satisfies \textbf{(A1.4)}
and the coefficients $(f,g)$ satisfy the following condition:
\begin{description}
\item[\bf(A1.5):] For any $t \leq T$, $f(t)$ and $g(t)$ are
$\mathscr{F}_{t}$-measurable with
$\frac{f(.)}{a}\in \mathscr{M}^{2}_{\beta }(\mathbf{R}^{k})$ and
$g(.)\in \mathscr{M}^{2}_{\beta }(\mathbf{R}^{k\times \ell })$.
\end{description}

Let us prove an a priori estimate of the solution in the following lemma.
%
\begin{lemma}%
\label{ll}
Let $(Y^{1},Z^{1},U^{1},K^{1},C^{1})$ and
$(Y^{2},Z^{2},U^{2},K^{2},C^{2})$ be two solutions to RBDSDEJs\index{RBDSDEJs} with parameters
$(f^{1}(.),g^{1}(.),\xi ^{1})$ and $(f^{2}(.),g^{2}(.),\xi ^{2})$, respectively.
We denote $\overline{\Re }:=\Re ^{1}-\Re ^{2}$ for
$\Re \in \{Y,Z,U,K,C,f,g,\xi \}$. Then there exists a constant
$\kappa (\beta )$ depending on $\beta $ such that for all $\beta>1$
\begin{eqnarray*}
&&\|\overline{Y}\|_{\mathscr{B}^{2}_{\beta }(\mathbf{R}^{k})}^{2}+\|
\overline{Z}\|_{\mathscr{M}^{2}_{\beta }(\mathbf{R}^{k\times d})}^{2}+
\|\overline{U}\|_{\mathscr{L}^{2}_{\beta }(\mathbf{R}^{k})}^{2}
\\
&\leq &\kappa (\beta ) \left (\|\overline{\xi }\|_{\mathscr{S}^{2}_{2
\beta }(\mathbf{R}^{k})}^{2}+\left \|  \frac{\overline{f}}{a}\right
\|  _{\mathscr{M}^{2}_{\beta }(\mathbf{R}^{k})}^{2} +\left \|
\overline{g}\right \|  _{\mathscr{M}^{2}_{\beta }(\mathbf{R}^{k\times
\ell })}^{2}\right ).
\end{eqnarray*}
\end{lemma}
\begin{proof}
Let $\tau \in \mathcal{T}_{[0,T]}$. It is obvious that the process
$\overline{Y}$ is an optional semimartingale\index{optional semimartingale} with the decomposition
$\overline{Y}_{\tau }=\overline{Y}_{0}+M_{\tau }+R_{\tau }+O_{\tau }$ where
$M_{\tau }=-\int _{0}^{\tau }\overline{g}(s)dB_{s}+\int _{0}^{\tau }
\overline{Z}_{s}dW_{s}+\int _{0}^{\tau }\int _{E}\overline{U}_{s}(e)
\widetilde{\mu }(ds,de)$,
$R_{\tau }=-\int _{0}^{\tau }\overline{f}(s)ds-\overline{K}_{\tau }$ and
$O_{\tau }=-\overline{C}_{\tau -}$. Then, from Lemma \ref{lem1}, we have
%
\begin{eqnarray}
\label{e4}
&&e^{\beta A_{t}}|\overline{Y}_{t}|^{2}+\beta \int _{t}^{T} e^{
\beta A_{s}} a^{2}_{s}|\overline{Y}_{s}|^{2}ds+\int _{t}^{T}e^{
\beta A_{s}}|\overline{Z}_{s}|^{2}ds
\nonumber
\\
&=&e^{\beta A_{T}}|\overline{\xi }_{T}|^{2}+ 2\int _{t}^{T}e^{\beta A_{s}}
\langle \overline{Y}_{s-},\overline{f}(s)\rangle ds+2\int _{t}^{T}e^{
\beta A_{s}}\langle \overline{Y}_{s-},d\overline{K}_{s}\rangle
\nonumber
\\
&&-2\int _{t}^{T}e^{\beta A_{s}}\langle \overline{Y}_{s-},
\overline{Z}_{s}dW_{s}\rangle -2\int _{t}^{T}\int _{E} e^{\beta A_{s}}
\langle \overline{Y}_{s-}, \overline{U}_{s}(e)\widetilde{\mu }(ds,de)
\rangle
\nonumber
\\
&&+2\int _{t}^{T}e^{\beta A_{s}}\langle \overline{Y}_{s-},
\overline{g}(s)dB_{s}\rangle +\int _{t}^{T}e^{\beta A_{s}}|
\overline{g}(s)|^{2}ds+2\int _{t}^{T}e^{\beta A_{s}}\langle
\overline{Y}_{s},d\overline{C}_{s}\rangle
\nonumber
\\
&&-\sum _{t<s\leq T}e^{\beta A_{s}}(\Delta \overline{Y}_{s})^{2}-
\sum _{t\leq s<T}e^{\beta A_{s}}(\Delta _{+} \overline{Y}_{s})^{2}.
\end{eqnarray}
From Remark \ref{rem02}, the processes $\overline{K}$ and $\mu $ do not
have jumps in common, but $\overline{K}$ jumps at predictable stopping
times\index{predictable stopping times} and $\mu $ jumps only at totally inaccessible stopping times.\index{totally inaccessible stopping times} Then
we can note that
\begin{equation*}
\sum _{t<s\leq T}e^{\beta A_{s}}(\Delta \overline{Y}_{s})^{2}=\int _{t}^{T}
\int _{E}e^{\beta A_{s}}|\overline{U}_{s}(e)|^{2}\mu (ds,de)+\sum _{t<s
\leq T}e^{\beta A_{s}}(\Delta \overline{K}_{s})^{2}.
\end{equation*}
Hence
\begin{eqnarray*}
&&\int _{t}^{T}e^{\beta A_{s}}\|\overline{U}_{s}\|^{2}_{\lambda }ds-
\sum _{t<s\leq T}e^{\beta A_{s}}(\Delta \overline{Y}_{s})^{2}
\\
&=&\int _{t}^{T}e^{\beta A_{s}}\|\overline{U}_{s}\|^{2}_{\lambda }ds-
\int _{t}^{T}\int _{E}e^{\beta A_{s}}|\overline{U}_{s}(e)|^{2}\mu (ds,de)-
\sum _{t<s\leq T}e^{\beta A_{s}}(\Delta \overline{K}_{s})^{2}
\\
&\leq &-\int _{t}^{T}\int _{E}e^{\beta A_{s}}|\overline{U}_{s}(e)|^{2}
\widetilde{\mu }(ds,de).
\end{eqnarray*}
On the other hand, by using the Skorokhod\index{Skorokhod condition} and minimality conditions\index{minimality condition} on
$\overline{K}$ and $\overline{C}$ we can show that
$\langle \overline{Y}_{s-},d\overline{K}_{s}\rangle \leq 0$ and
$\langle \overline{Y}_{s},d\overline{C}_{s}\rangle \leq 0$. Indeed, for
all $s\leq T$
\begin{eqnarray*}
\langle \overline{Y}_{s-},d\overline{K}_{s}\rangle &=&\langle Y^{1}_{s-}-
\xi _{s-},dK^{1,c}_{s}+\Delta K^{1,d}_{s}\rangle -\langle Y^{2}_{s-}-
\xi _{s-},dK^{1,c}_{s}+\Delta K^{1,d}_{s}\rangle
\\
&&-\langle Y^{1}_{s-}-\xi _{s-},dK^{2,c}_{s}+\Delta K^{2,d}_{s}
\rangle +\langle Y^{2}_{s-}-\xi _{s-},dK^{2,c}_{s}+\Delta K^{2,d}_{s}
\rangle
\\
&=&-\langle Y^{2}_{s-}-\xi _{s-},dK^{1,c}_{s}+\Delta K^{1,d}_{s}
\rangle -\langle Y^{1}_{s-}-\xi _{s-},dK^{2,c}_{s}+\Delta K^{2,d}_{s}
\rangle
\\
&\leq &0, \quad \text{ since }Y^{i}\geq \xi \text{ for }i=1,2.
\end{eqnarray*}
Furthermore we have
$\langle \overline{Y}_{s},d\overline{C}_{s}\rangle =\langle
\overline{Y}_{s},\Delta \overline{C}_{s}\rangle $, and by the same arguments,
we have, for all $s\leq T$,
\begin{eqnarray*}
\langle \overline{Y}_{s},\Delta \overline{C}_{s}\rangle &=&\langle Y^{1}_{s}-
\xi _{s},\Delta C^{1}_{s}\rangle -\langle Y^{2}_{s}-\xi _{s},\Delta C^{1}_{s}
\rangle -\langle Y^{1}_{s}-\xi _{s},\Delta C^{2}_{s}\rangle
\\
&&-\langle \xi _{s}-Y^{2}_{s},\Delta C^{2}_{s}\rangle
\\
&=&0-\langle Y^{2}_{s}-\xi _{s},\Delta C^{1}_{s}\rangle -\langle Y^{1}_{s}-
\xi _{s},\Delta C^{2}_{s}\rangle -0
\\
&\leq &0, \quad \text{ since }Y^{i}\geq \xi \text{ for }i=1,2.
\end{eqnarray*}
Moreover, by using the fact that
\begin{equation*}
2\langle \overline{Y}_{s},\overline{f}(s)\rangle \leq (\beta -1)a_{s}^{2}|
\overline{Y}_{s}|^{2}+\frac{1}{\beta -1}
\frac{|\overline{f}(s)|^{2}}{a_{s}^{2}} \quad \forall \beta >1,
\end{equation*}
the inequality \eqref{e4} becomes
%
\begin{eqnarray}
\label{e5}
&&e^{\beta A_{t}}|\overline{Y}_{t}|^{2}+\int _{t}^{T} e^{\beta A_{s}}
a^{2}_{s}|\overline{Y}_{s}|^{2}ds+\int _{t}^{T}e^{\beta A_{s}}|
\overline{Z}_{s}|^{2} ds+\int _{t}^{T}e^{\beta A_{s}}\|\overline{U}_{s}
\|^{2}_{\lambda }ds
\nonumber
\\
&\leq &\operatorname*{ess
\,sup}_{\tau \in \mathcal{T}_{[0,T]}}e^{2\beta A_{\tau }}|
\overline{\xi }_{\tau }|^{2}+\frac{1}{\beta -1}\int _{t}^{T}e^{
\beta A_{s}}\left |\frac{\overline{f}(s)}{a_{s}}\right |^{2}ds-2
\int _{t}^{T}e^{\beta A_{s}}\langle \overline{Y}_{s-},\overline{Z}_{s}dW_{s}
\rangle
\nonumber
\\
&&-2\int _{t}^{T}\int _{E} e^{\beta A_{s}}\langle \overline{Y}_{s-},
\overline{U}_{s}(e)\widetilde{\mu }(ds,de)\rangle +2\int _{t}^{T}e^{
\beta A_{s}}\langle \overline{Y}_{s-},\overline{g}(s)dB_{s}\rangle
\nonumber
\\
&&+\int _{t}^{T}e^{\beta A_{s}}|\overline{g}(s)|^{2}ds.
\end{eqnarray}
Taking the expectation on the both sides of the inequality (\ref{e5}) and
using Proposition \ref{prop}, we get, for all $\beta >1$,
%
\begin{eqnarray}
\label{e6}
&&\|\overline{Y}\|_{\mathscr{M}^{2,a}_{\beta }(\mathbf{R}^{k})}^{2}+
\|\overline{Z}\|_{\mathscr{M}^{2}_{\beta }(\mathbf{R}^{k\times d})}^{2}+
\|\overline{U}\|_{\mathscr{L}^{2}_{\beta }(\mathbf{R}^{k})}^{2}
\nonumber
\\
&\leq & \|\overline{\xi }\|_{\mathscr{S}^{2}_{2\beta }(\mathbf{R}^{k})}^{2}+
\frac{1}{\beta -1}\left \|  \frac{\overline{f}}{a}\right \|  _{
\mathscr{M}^{2}_{\beta }(\mathbf{R}^{k})}^{2} +\left \|  \overline{g}
\right \|  _{\mathscr{M}^{2}_{\beta }(\mathbf{R}^{k\times \ell })}^{2}.
\end{eqnarray}
On the other hand, by taking the essential supremum over
$\tau \in \mathcal{T}_{[0,T]}$ and then the expectation on both sides of
inequality (\ref{e5}) we obtain
\begin{eqnarray*}
&&{\mathbf{E}}\operatorname*{ess
\,sup}_{\tau \in \mathcal{T}_{[0,T]}}e^{\beta A_{\tau }}|\overline{Y}_{
\tau }|^{2}
\\
&&%
\hspace{-0.75cm}%
\leq {\mathbf{E}}\operatorname*{ess
\,sup}_{\tau \in \mathcal{T}_{[0,T]}}e^{2\beta A_{\tau }}|
\overline{\xi }_{\tau }|^{2}+\frac{1}{\beta -1}{\mathbf{E}}\int _{0}^{T}e^{
\beta A_{s}}\left |\frac{\overline{f}(s)}{a_{s}}\right |^{2}ds
\\
&&+2{\mathbf{E}}\operatorname*{ess
\,sup}_{\tau \in \mathcal{T}_{[0,T]}}\left |\int _{0}^{\tau }e^{\beta A_{s}}
\langle \overline{Y}_{s-},\overline{Z}_{s}dW_{s}\rangle \right |
\\
&&+2{\mathbf{E}}\operatorname*{ess
\,sup}_{\tau \in \mathcal{T}_{[0,T]}}\left |\int _{0}^{\tau }\int _{E} e^{
\beta A_{s}}\langle \overline{Y}_{s-}, \overline{U}_{s}(e)
\widetilde{\mu }(ds,de)\rangle \right |
\\
&&+2{\mathbf{E}}\operatorname*{ess
\,sup}_{\tau \in \mathcal{T}_{[0,T]}}\left |\int _{0}^{\tau }e^{\beta A_{s}}
\langle \overline{Y}_{s-},\overline{g}(s)dB_{s}\rangle \right | +
\int _{0}^{T}e^{\beta A_{s}}|\overline{g}(s)|^{2}ds.
\end{eqnarray*}
From the Burkh\"{o}lder--Davis--Gundy inequality, there exists a universal
constant $c$ such that
\begin{eqnarray*}
2{\mathbf{E}}\operatorname*{ess
\,sup}_{\tau \in \mathcal{T}_{[0,T]}}\left |\int _{0}^{\tau }e^{\beta A_{s}}
\langle \overline{Y}_{s-},\overline{Z}_{s}dW_{s}\rangle \right | &
\leq & 2c{\mathbf{E}}\sqrt{\int _{0}^{T} e^{2\beta A_{s}}|\overline{Y}_{s-}|^{2}|
\overline{Z}_{s}|^{2}ds}
\\
&\leq &\frac{1}{4}\|\overline{Y}\|_{\mathscr{S}^{2}_{\beta }(
\mathbf{R}^{k})}^{2}+4c^{2}\|\overline{Z}\|_{\mathscr{M}^{2}_{\beta }(
\mathbf{R}^{k\times d})}^{2},
\end{eqnarray*}
\begin{eqnarray*}
&&2{\mathbf{E}}\operatorname*{ess
\,sup}_{\tau \in \mathcal{T}_{[0,T]}}\left |\int _{0}^{\tau }\int _{E} e^{
\beta A_{s}}\langle \overline{Y}_{s-}, \overline{U}_{s}(e)
\widetilde{\mu }(ds,de)\rangle \right |
\\
&\leq & 2c{\mathbf{E}}\sqrt{\int _{0}^{T} e^{2\beta A_{s}}|\overline{Y}_{s-}|^{2}
\|\overline{U}_{s}\|_{\lambda }^{2}ds} \;\;\leq \;\;\frac{1}{4}\|
\overline{Y}\|_{\mathscr{S}^{2}_{\beta }(\mathbf{R}^{k})}^{2}+4c^{2}
\|\overline{U}\|_{\mathscr{L}^{2}_{\beta }(\mathbf{R}^{k})}^{2}
\end{eqnarray*}
and
\begin{eqnarray*}
2{\mathbf{E}}\operatorname*{ess
\,sup}_{\tau \in \mathcal{T}_{[0,T]}}\left |\int _{0}^{\tau }e^{\beta A_{s}}
\langle \overline{Y}_{s-},\overline{g}(s)dB_{s}\rangle \right | &
\leq & 2c{\mathbf{E}}\sqrt{\int _{0}^{T} e^{2\beta A_{s}}|\overline{Y}_{s-}|^{2}|
\overline{g}(s)|^{2}ds}
\\
&\leq &\frac{1}{4}\|\overline{Y}\|_{\mathscr{S}^{2}_{\beta }(
\mathbf{R}^{k})}^{2}+4c^{2}\|\overline{g}\|_{\mathscr{M}^{2}_{\beta }(
\mathbf{R}^{k\times \ell })}^{2}.
\end{eqnarray*}
Consequently,
%
\begin{eqnarray}
\label{e7}
\|\overline{Y}\|_{\mathscr{S}^{2}_{\beta }(\mathbf{R}^{k})}^{2} &
\leq &4\left (\|\overline{\xi }\|_{\mathscr{S}^{2}_{2\beta }(
\mathbf{R}^{k})}^{2}+\frac{1}{\beta -1}\left \|
\frac{\overline{f}}{a}\right \|  _{\mathscr{M}^{2}_{\beta }(\mathbf{R}^{k})}^{2}
+(4c^{2}+1)\left \|  \overline{g}\right \|  _{\mathscr{M}^{2}_{\beta }(
\mathbf{R}^{k\times \ell })}^{2}\right .
\nonumber
\\
&&\left .+4c^{2}\|\overline{Z}\|_{\mathscr{M}^{2}_{\beta }(\mathbf{R}^{k
\times d})}^{2} +4c^{2}\|\overline{U}\|_{\mathscr{L}^{2}_{\beta }(
\mathbf{R}^{k})}^{2}\right ).
\end{eqnarray}
The desired result is obtained by combining the estimates (\ref{e6}) and (\ref{e7})
for $\beta >1$.
\end{proof}

In the following, we state the existence and uniqueness result for the solution
to\break  RBDSDEJ \eqref{eq1}.

\begin{proposition}
\label{pro1}
Under the assumptions \emph{\textbf{(A1.4)}} and \emph{\textbf{(A1.5)}}, the RBDSDEJ \eqref{eq1} admits a unique solution
$(Y,Z,U,K,C)\in \mathscr{B}^{2}_{\beta }(\mathbf{R}^{k})\times
\mathscr{S}^{2}(\mathbf{R}^{k})\times \mathscr{S}^{2}(\mathbf{R}^{k})$ for all $\beta>1$,
and for each $\nu \in \mathcal{T}_{[0,T]}$ we have
\begin{equation*}
Y_{\nu }=\operatorname*{ess
\,sup}_{\tau \in \mathcal{T}_{[\nu ,T]}}{\mathbf{E}}\left [\xi _{\tau }+
\int _{\nu }^{\tau }f(t)dt +\int _{\nu }^{\tau }g(t)dB_{t}| \mathcal{G}_{\nu
}\right ]\quad a.s.
\end{equation*}
\end{proposition}

\begin{proof}
Let $\nu \in \mathcal{T}_{[0,T]}$. We define the value function
$\overline{\overline{Y}}(\nu )$ by
\begin{equation*}
\overline{\overline{Y}}(\nu )=\operatorname*{ess
\,sup}_{\tau \in \mathcal{T}_{[\nu ,T]}}{\mathbf{E}}\left [\xi _{\tau }+
\int _{\nu }^{\tau }f(t)dt +\int _{\nu }^{\tau }g(t)dB_{t}| \mathcal{G}_{\nu
}\right ],
\end{equation*}
and $\widetilde{Y}(\nu )$ by
\begin{eqnarray*}
\widetilde{Y}(\nu )&=&\overline{\overline{Y}}(\nu )+\int _{0}^{\nu }f(t)dt
+\int _{0}^{\nu }g(t)dB_{t}
\\
&=&\operatorname*{ess
\,sup}_{\tau \in \mathcal{T}_{[\nu ,T]}}{\mathbf{E}}\left [\xi _{\tau }+
\int _{0}^{\tau }f(t)dt+\int _{0}^{\tau }g(t)dB_{t}| \mathcal{G}_{\nu
}\right ].
\end{eqnarray*}
The process
$\left (\xi _{t}+\int _{0}^{t} f(s)ds +\int _{0}^{t} g(s)dB_{s}
\right )_{t\leq T}$ is progressively measurable. Therefore, the family
$(\widetilde{Y}(\nu ))_{\nu \in \mathcal{T}_{[0,T]}}$ is a supermartingale
family. This observation with the Remark \textbf{b.} page 435 in
\cite{DM} ensures the existence of a strong optional supermartingale
$\widetilde{Y}$\index{optional supermartingale} such that
$\widetilde{Y}_{\nu }=\widetilde{Y}(\nu )$ for all
$\nu \in \mathcal{T}_{[0,T]}$. Thus, we have
$\overline{\overline{Y}}(\nu )=\widetilde{Y}_{\nu }-\int _{0}^{\nu }f(t)dt
-\int _{0}^{\nu }g(t)dB_{t}$. On the other hand, almost all trajectories
of the strong optional supermartingale\index{optional supermartingale} are l\`{a}dl\`{a}g, then the l\`{a}dl\`{a}g
optional process
$(\overline{\overline{Y}}_{t})_{t\leq T}:=\left (\widetilde{Y}_{t}-
\int _{0}^{t} f(s)ds -\int _{0}^{t} g(s)dB_{s}\right )_{t\leq T}$ aggregates
the family
$(\overline{\overline{Y}}(\nu ))_{\nu \in \mathcal{T}_{[0,T]}}$.

Now, it remains to show that the candidate
$\overline{\overline{Y}}\in \mathscr{S}^{2}_{\beta }(\mathbf{R}^{k})$.
Using the Jensen's, Young's and H\"{o}lder's inequalities respectively, we
obtain
\begin{eqnarray*}
&&e^{\frac{\beta }{2}A_{\nu }}|\overline{\overline{Y}}_{\nu }|
\\
&=&\Biggl |\operatorname*{ess
\,sup}_{\tau \in \mathcal{T}_{[\nu ,T]}}{\mathbf{E}}\Biggl [e^{
\frac{\beta }{2}A_{\nu }}\xi _{\tau }+e^{\frac{\beta }{2}A_{\nu }}\int _{\nu }^{\tau }f(t)dt +e^{\frac{\beta }{2}A_{\nu }}\int _{\nu }^{\tau }g(t)dB_{t}|
\mathcal{G}_{\nu }\Biggr ]\Biggr |
\\
&\leq &{\mathbf{E}}\Biggl [\Biggl \{  \Biggl |\operatorname*{ess
\,sup}_{\tau \in \mathcal{T}_{[\nu ,T]}}e^{\frac{\beta }{2}A_{\nu }}
\xi _{\tau }+e^{\frac{\beta }{2}A_{\nu }}\int _{\nu }^{T} f(t)dt
\\
&&
\hspace{4.5cm}
+\operatorname*{ess
\,sup}_{\tau \in \mathcal{T}_{[\nu ,T]}}e^{\frac{\beta }{2}A_{\nu }}
\int _{\nu }^{\tau }g(t)dB_{t}\Biggr |^{2}\Biggr \}  ^{\frac{1}{2}}|
\mathcal{G}_{\nu }\Biggr ]
\\
&\leq &\sqrt{3}{\mathbf{E}}\Biggl [\Biggl \{  \operatorname*{ess
\,sup}_{\tau \in \mathcal{T}_{[\nu ,T]}}e^{\beta A_{\tau }}|\xi _{\tau }|^{2}+e^{
\beta A_{\nu }}\Biggl |\int _{\nu }^{T} f(t)dt\Biggr |^{2}
\\
&&
\hspace{4.5cm}
+\operatorname*{ess
\,sup}_{\tau \in \mathcal{T}_{[\nu ,T]}}e^{\beta A_{\nu }}\Biggl |\int _{\nu }^{\tau }g(t)dB_{t}\Biggr |^{2}\Biggr \}  ^{\frac{1}{2}}|\mathcal{G}_{\nu }\Biggr ]
\\
&\leq &\sqrt{3}{\mathbf{E}}\Biggl [\Biggl \{  \operatorname*{ess
\,sup}_{\tau \in \mathcal{T}_{[\nu ,T]}}e^{2\beta A_{\tau }}|\xi _{\tau }|^{2}+e^{
\beta A_{\nu }}\Biggl (\int _{\nu }^{T}e^{-\beta A_{t}}a_{t}^{2}dt\Biggr )
\times
\\
&&
\hspace{2.5cm}
\Biggl (\int _{\nu }^{T}e^{\beta A_{t}} \Biggl |\frac{f(t)}{a_{t}}\Biggr |^{2}dt
\Biggr )+c\int _{0}^{T}e^{\beta A_{t}} |g(t)|^{2}dt\Biggr \}  ^{\frac{1}{2}}|\mathcal{G}_{\nu }\Biggr ]
\\
&\leq &\sqrt{3}{\mathbf{E}}\Biggl [\Biggl \{  \operatorname*{ess
\,sup}_{\tau \in \mathcal{T}_{[0,T]}}e^{2\beta A_{\tau }}|\xi _{\tau }|^{2}+
\frac{1}{\beta }\int _{0}^{T} e^{\beta A_{t}}\Biggl |
\frac{f(t)}{a_{t}}\Biggr |^{2}dt
\\
&&
\hspace{6cm}
+c\int _{0}^{T}e^{\beta A_{t}} |g(t)|^{2}dt\Biggr \}  ^{\frac{1}{2}}|
\mathcal{G}_{\nu }\Biggr ].
\end{eqnarray*}
Taking the essential supremum over $\nu \in \mathcal{T}_{[0,T]}$ on the
above sides and using the Doob's martingale inequality, we conclude that
\begin{eqnarray*}
&&{\mathbf{E}}\operatorname*{ess
\,sup}_{\nu \in \mathcal{T}_{[0,T]}} e^{\beta A_{\nu }}|
\overline{\overline{Y}}_{\nu }|^{2}
\\
&\leq & \kappa '(\beta ){\mathbf{E}}\left ( \operatorname*{ess
\,sup}_{\tau \in \mathcal{T}_{[\nu ,T]}}e^{2\beta A_{\tau }}|\xi _{\tau }|^{2}+
\int _{0}^{T} e^{\beta A_{t}}\left |\frac{f(t)}{a_{t}}\right |^{2}dt+
\int _{0}^{T}e^{\beta A_{t}} |g(t)|^{2}dt\right )
\end{eqnarray*}
where $\kappa '(\beta )$ is a positive constant depending on $\beta $. It
follows that
$\overline{\overline{Y}}\in \mathscr{S}^{2}_{\beta }(\mathbf{R}^{k})$.

Note that the strong optional supermartingale $\widetilde{Y}$\index{optional supermartingale} is of class
(D) (i.e. the set of all random variables $\widetilde{Y}_{\nu }$, for each
finite stopping time $\nu $, is uniformly integrable). Then by the Mertens
decomposition\index{Mertens decomposition} (see Theorem \ref{ap1}), there exists a uniformly integrable
martingale (c\`{a}dl\`{a}g) $N$, a nondecreasing right-continuous predictable
process $K$ (with $K_{0}=0$) such that ${\mathbf{E}}|K_{T}|^{2}<+\infty $ and
a nondecreasing right-continuous adapted purely discontinuous process
$C$ (with $C_{0-}=0$) such that ${\mathbf{E}}|C_{T}|^{2}<+\infty $, 
with the following equality:
\begin{equation*}
\widetilde{Y}_{\tau }=N_{\tau }-K_{\tau }-C_{\tau -}\quad \forall \tau
\in \mathcal{T}_{[0,T]}.
\end{equation*}
By an extension of It\^{o}'s martingale representation Theorem, there exists
a unique pair of predictable processes
$(Z,U)\in \mathscr{M}^{2}(\mathbf{R}^{k\times d})\times \mathscr{L}^{2}(
\mathbf{R}^{k})$ such that
\begin{equation*}
N_{\tau }=N_{0}+\int _{0}^{\tau }Z_{s}dW_{s}+\int _{0}^{\tau }\int _{E} U_{s}(e)
\widetilde{\mu }(ds,de).
\end{equation*}
Hence for each $\tau \in \mathcal{T}_{[0,T]}$
%
\begin{eqnarray}
\label{e9}
\overline{\overline{Y}}_{\tau }&=&-\int _{0}^{\tau }f(s)ds-\int _{0}^{\tau }g(s)dB_{s}+N_{0}+\int _{0}^{\tau }Z_{s}dW_{s}
\nonumber
\\
&&-\int _{0}^{\tau }\int _{E} U_{s}(e)\widetilde{\mu }(ds,de)-K_{\tau }-C_{
\tau -}
\end{eqnarray}
with
$\overline{\overline{Y}}_{T}=\overline{\overline{Y}}(T)=\xi _{T}$ and
$\overline{\overline{Y}}_{\tau }=\overline{\overline{Y}}(\tau )
\geq \xi _{\tau }$ a.s for all $\tau \in \mathcal{T}_{[0,T]}$. Next, let
us focus on the Skorokhod\index{Skorokhod condition} and minimality conditions\index{minimality condition}. Since
$\Delta _{+}\overline{\overline{Y}}_{\tau }=\mathbh{1}_{\{
\overline{\overline{Y}}_{\tau }=\xi _{\tau }\}}\Delta _{+}
\overline{\overline{Y}}_{\tau }$ a.s.(see Remark A.4 in
\cite{GIOOQ:2017}), 
from (\ref{e9}) we have
$\Delta C_{\tau }=-\Delta _{+}\overline{\overline{Y}}_{\tau }$ a.s., then
$\Delta C_{\tau }=\mathbh{1}_{\{\overline{\overline{Y}}_{\tau }=\xi _{\tau }\}}\Delta C_{\tau }$ a.s. It follows that the minimality condition\index{minimality condition} on
$C$ is satisfied. Further, 
due to a result 
from the optimal stopping
theory (see Proposition B.11 in \cite{KQ:2012}), for each predictable stopping
time $\tau $, we have
$\int _{0}^{T}\mathbh{1}_{\{\overline{\overline{Y}}_{t}>\xi _{t}\}}dK^{c}_{t}=0$
a.s. and
$\Delta K^{d}_{\tau }=\mathbh{1}_{\{\overline{\overline{Y}}_{\tau -}=
\xi _{\tau -}\}}\Delta K^{d}_{\tau }$ a.s. Then the process $K$ satisfies
the Skorokhod condition.\index{Skorokhod condition} 
Thus, we found a process
$(\overline{\overline{Y}},Z,U,K,C)$ which satisfies the RBDSDEJ \eqref{eq1}.

Now, it remains to show that
$(\overline{\overline{Y}},Z,U,K,C)\in \mathscr{B}^{2}_{\beta }(
\mathbf{R}^{k})\times \mathscr{S}^{2}(\mathbf{R}^{k})\times
\mathscr{S}^{2}(\mathbf{R}^{k})$. Indeed, let
$\widetilde{K}_{t}:=K_{t}+C_{t-}$  
be the Mertens process\index{Mertens process} associated with
$\widetilde{Y}$. By the definition of $\widetilde{Y}_{\nu }$, we see that
\begin{eqnarray*}
|\widetilde{Y}_{\nu }|&=&\left |\operatorname*{ess
\,sup}_{\tau \in \mathcal{T}_{[\nu ,T]}}{\mathbf{E}}\left [\xi _{\tau }+
\int _{0}^{\tau }f(t)dt +\int _{0}^{\tau }g(t)dB_{t}|\mathcal{G}_{\nu
}\right ]\right |
\\
&\leq &{\mathbf{E}}\left [\operatorname*{ess
\,sup}_{\tau \in \mathcal{T}_{[\nu ,T]}}|\xi _{\tau }|+\int _{0}^{T} |f(t)|dt+
\operatorname*{ess
\,sup}_{\tau \in \mathcal{T}_{[\nu ,T]}}\left |\int _{0}^{\tau }g(t)dB_{t}
\right ||\mathcal{G}_{\nu }\right ].
\end{eqnarray*}
From Corollary \ref{ap2}, there exists a positive constant $c$ such
that
\begin{eqnarray*}
{\mathbf{E}}|\widetilde{K}_{T}|^{2}&\leq &c{\mathbf{E}}\left |
\operatorname*{ess
\,sup}_{\tau \in \mathcal{T}_{[\nu ,T]}}|\xi _{\tau }|+\int _{0}^{T} |f(t)|dt+
\operatorname*{ess
\,sup}_{\tau \in \mathcal{T}_{[\nu ,T]}}\left |\int _{0}^{\tau }g(t)dB_{t}
\right |\right |^{2}
\\
&\leq & c(\beta )\left (\|\xi \|_{\mathscr{S}^{2}_{2\beta }(
\mathbf{R}^{k})}^{2}+\left \|  \frac{f}{a}\right \|  _{\mathscr{M}^{2}_{\beta }(\mathbf{R}^{k})}^{2} +\left \|  g\right \|  _{\mathscr{M}^{2}_{\beta }(\mathbf{R}^{k\times \ell })}^{2}\right )
\end{eqnarray*}
where $c(\beta )$ is a positive constant depending on $\beta $. 
$\widetilde{K}$ is nondecreasing, and it implies that
\begin{equation*}
{\mathbf{E}}\operatorname*{ess
\,sup}_{\nu \in \mathcal{T}_{[0,T]}}|\widetilde{K}_{\tau }|^{2}\leq {
\mathbf{E}}|\widetilde{K}_{T}|^{2}<+\infty .
\end{equation*}
It follows that $\widetilde{K}\in \mathscr{S}^{2}(\mathbf{R}^{k})$, then
$(K,C)\in \mathscr{S}^{2}(\mathbf{R}^{k})\times \mathscr{S}^{2}(
\mathbf{R}^{k})$. On the other hand, from Lemma \ref{lem1} we have
\begin{eqnarray*}
&&e^{\beta A_{t}}|\overline{\overline{Y}}_{t}|^{2}+\beta \int _{t}^{T}
e^{\beta A_{s}} a^{2}_{s}|\overline{\overline{Y}}_{s}|^{2}ds+\int _{t}^{T}e^{
\beta A_{s}}|Z_{s}|^{2}ds
\nonumber
\\
&=&e^{\beta A_{T}}|\xi _{T}|^{2}+ 2\int _{t}^{T}e^{\beta A_{s}}
\langle \overline{\overline{Y}}_{s-},f(s)\rangle ds+2\int _{t}^{T}e^{
\beta A_{s}}\langle \overline{\overline{Y}}_{s-},dK_{s}\rangle
\nonumber
\\
&&-2\int _{t}^{T}e^{\beta A_{s}}\langle \overline{\overline{Y}}_{s-},Z_{s}dW_{s}
\rangle -2\int _{t}^{T}\int _{E} e^{\beta A_{s}}\langle
\overline{\overline{Y}}_{s-}, U_{s}(e)\widetilde{\mu }(ds,de)
\rangle
\nonumber
\\
&&+2\int _{t}^{T}e^{\beta A_{s}}\langle \overline{\overline{Y}}_{s-},g(s)dB_{s}
\rangle +\int _{t}^{T}e^{\beta A_{s}}|g(s)|^{2}ds+2\int _{t}^{T}e^{
\beta A_{s}}\langle \overline{\overline{Y}}_{s},dC_{s}\rangle
\nonumber
\\
&&-\sum _{t<s\leq T}e^{\beta A_{s}}(\Delta \overline{\overline{Y}}_{s})^{2}-
\sum _{t\leq s<T}e^{\beta A_{s}}(\Delta _{+}
\overline{\overline{Y}}_{s})^{2}.
\end{eqnarray*}
From Remark \ref{rem02}, the processes $K$ and $\mu $ do not have jumps
in common, but $K$ jumps at predictable stopping times\index{predictable stopping times} and $\mu $ jumps
only at totally inaccessible stopping times,\index{totally inaccessible stopping times} then we can write
\begin{equation*}
\sum _{t<s\leq T}e^{\beta A_{s}}(\Delta \overline{\overline{Y}}_{s})^{2}=
\int _{t}^{T}\int _{E}e^{\beta A_{s}}|U_{s}(e)|^{2}\mu (ds,de)+\sum _{t<s
\leq T}e^{\beta A_{s}}(\Delta K_{s})^{2}.
\end{equation*}
Hence
\begin{eqnarray*}
&&\int _{t}^{T}e^{\beta A_{s}}\|U_{s}\|^{2}_{\lambda }ds-\sum _{t<s
\leq T}e^{\beta A_{s}}(\Delta \overline{\overline{Y}}_{s})^{2}
\\
&=&\int _{t}^{T}e^{\beta A_{s}}\|U_{s}\|^{2}_{\lambda }ds-\int _{t}^{T}
\int _{E}e^{\beta A_{s}}|U_{s}(e)|^{2}\mu (ds,de)-\sum _{t<s\leq T}e^{
\beta A_{s}}(\Delta K_{s})^{2}
\\
&\leq &-\int _{t}^{T}\int _{E}e^{\beta A_{s}}|U_{s}(e)|^{2}
\widetilde{\mu }(ds,de).
\end{eqnarray*}
Consequently,
\begin{eqnarray*}
&&\int _{t}^{T} e^{\beta A_{s}} a^{2}_{s}|\overline{\overline{Y}}_{s}|^{2}ds+
\int _{t}^{T}e^{\beta A_{s}}|Z_{s}|^{2}ds+\int _{t}^{T}e^{\beta A_{s}}
\|U_{s}\|^{2}_{\lambda }ds
\nonumber
\\
&\leq &e^{\beta A_{T}}|\xi _{T}|^{2}+ \frac{1}{\beta -1}\int _{t}^{T}e^{
\beta A_{s}}\left |\frac{f(s)}{a_{s}}\right |^{2}ds-2\int _{t}^{T}e^{
\beta A_{s}}\langle \overline{\overline{Y}}_{s-},Z_{s}dW_{s}
\rangle
\nonumber
\\
&&-2\int _{t}^{T}\int _{E} e^{\beta A_{s}}\langle
\overline{\overline{Y}}_{s-}, U_{s}(e)\widetilde{\mu }(ds,de)
\rangle +2\int _{t}^{T}e^{\beta A_{s}}\langle
\overline{\overline{Y}}_{s-},g(s)dB_{s}\rangle
\nonumber
\\
&&+\int _{t}^{T}e^{\beta A_{s}}|g(s)|^{2}ds+2
\operatorname*{ess
\,sup}_{\tau \in \mathcal{T}_{[0,T]}}e^{2\beta A_{\tau }}|\xi _{\tau }|^{2}+K_{T}^{2}+C_{T}^{2}.
\end{eqnarray*}
Here we have used also the Skorokhod\index{Skorokhod condition} and minimality conditions\index{minimality condition} on
$K$ and $C$. Next, by taking the expectation on both sides of above inequality,
we get
\begin{eqnarray*}
&&\|\overline{\overline{Y}}\|_{\mathscr{M}^{2,a}_{\beta }(\mathbf{R}^{k})}^{2}+
\|Z\|_{\mathscr{M}^{2}_{\beta }(\mathbf{R}^{k\times d})}^{2}+\|U\|_{
\mathscr{L}^{2}_{\beta }(\mathbf{R}^{k})}^{2}
\\
&\leq & 3\|\xi \|_{\mathscr{S}^{2}_{2\beta }(\mathbf{R}^{k})}^{2}+
\frac{1}{\beta -1}\left \|  \frac{f}{a}\right \|  _{\mathscr{M}^{2}_{\beta }(\mathbf{R}^{k})}^{2} +\left \|  g\right \|  _{\mathscr{M}^{2}_{\beta }(\mathbf{R}^{k\times \ell })}^{2}+{\mathbf{E}}|K_{T}|^{2}+{\mathbf{E}}|C_{T}|^{2}.
\end{eqnarray*}
Then
$(\overline{\overline{Y}},Z,U)\in \mathscr{M}^{2,a}_{\beta }(
\mathbf{R}^{k})\times \mathscr{M}^{2}_{\beta }(\mathbf{R}^{k\times d})
\times \mathscr{L}^{2}_{\beta }(\mathbf{R}^{k})$.

Finally, it is remarkabe that the uniqueness of the solution comes from the
uniqueness of the Mertens decomposition\index{Mertens decomposition} and the It\^{o}'s martingale representation
Theorem, and if $\overline{\overline{Y}}$ and $Y$ are two first-components
of the solution, then by Lemma \ref{ll} we have immediately
$\overline{\overline{Y}}=Y$.
\end{proof}

\begin{proposition}
\label{pro2}
Assume that the assumptions \emph{\textbf{(A1.1)}}--\emph{\textbf{(A1.4)}} are true. Then, if
$(y,z,\break u)\in \mathscr{B}^{2}_{\beta }(\mathbf{R}^{k})$ for $\beta>1$, there exists a unique
process
$(Y,Z,U,K,C)\in \mathscr{B}^{2}_{\beta }(\mathbf{R}^{k})\times
\mathscr{S}^{2}(\mathbf{R}^{k})\times \mathscr{S}^{2}(\mathbf{R}^{k})$
being a solution to the following RBDSDEJ, for all
$\tau \in \mathcal{T}_{[0,T]}$,
\begin{equation*}
\left \{
\begin{array}{l}
Y_{\tau }=\xi _{T} +\displaystyle \int _{\tau }^{T} f(s,y_{s},z_{s},u_{s})ds+
\int _{\tau }^{T} g(s,y_{s},z_{s},u_{s})dB_{s}-\int _{\tau }^{T}Z_{s}dW_{s}

\\
\hspace{1cm}
-\displaystyle \int _{\tau }^{T}\int _{E}U_{s}(e)\widetilde{\mu }(ds,de)+K_{T}-K_{
\tau }+C_{T-}-C_{\tau -},
\\
Y_{\tau }\geq \xi _{\tau },
\\
\displaystyle \int _{0}^{T}\mathbh{1}_{\{Y_{t}>\xi _{t}\}}dK^{c}_{t}=0,\quad
 (Y_{\tau -}-\xi _{\tau -})\Delta K^{d}_{\tau }=0\quad  \text{and}\quad (Y_{\tau }-
\xi _{\tau })\Delta C_{\tau }=0 \; a.s.
\end{array}
\right .
\end{equation*}
\end{proposition}

\begin{proof}
Given $(y,z,u)\in \mathscr{B}^{2}_{\beta }(\mathbf{R}^{k})$, we define
$\widehat{f}(t) = f(t, y_{t}, z_{t}, u_{t})$ and
$\widehat{g}(t) = g(t, y_{t}, z_{t}, u_{t})$. Let us show that
$\widehat{f}$ and $\widehat{g}$ satisfy \textbf{(A1.5)}. From the assumptions
\textbf{(A1.1)} and \textbf{(A1.2)}, we have
\begin{equation*}
|\widehat{f}(s)|^{2} \leq 4\left (a^{4}_{s}|y_{s}|^{2} +a^{2}_{s}|z_{s}|^{2}+a^{2}_{s}
\|u_{s}\|_{\lambda }^{2}+|f(s,0)|^{2}\right )
\end{equation*}
and
\begin{equation*}
|\widehat{g}(s)|^{2}\le 2\left (a^{2}_{s}|y_{s}|^{2}+\alpha (|z_{s}|^{2}+
\|u_{s}\|_{\lambda }^{2})+|g(s, 0)|^{2}\right ).
\end{equation*}
Thus gathering these inequalities, we deduce that
\begin{eqnarray*}
&& {\mathbf{E}}\bigg (\int _{0}^{T}e^{\beta A_{s}}\left |
\frac{\widehat{f}(s)}{a_{s}}\right |^{2}ds +\int _{0}^{T}e^{\beta A_{s}}|
\widehat{g}(s)|^{2}ds\bigg )
\\
&\leq & {\mathbf{E}}\left (6\int _{0}^{T}e^{\beta A_{s}}a^{2}_{s}|y_{s}|^{2}ds+(4+2
\alpha )\int _{0}^{T}e^{\beta A_{s}}(|z_{s}|^{2} + \|u_{s}\|_{\lambda }^{2})ds
\right )
\\
&&+{\mathbf{E}}\left (4\int _{0}^{T}e^{\beta A_{s}}\left |
\frac{f(s,0)}{a_{s}}\right |^{2}ds + 2\int _{0}^{T}e^{\beta A_{s}}|g(s,0)|^{2}ds
\right ).
\end{eqnarray*}
This implies that $\widehat{f}$ and $\widehat{g}$ satisfy \textbf{(A1.5)}
since $(y,z,u)\in \mathscr{B}^{2}_{\beta }(\mathbf{R}^{k})$ and in view
of the assumption \textbf{(A1.3)}. Hence the result follows from Proposition \ref{pro1}.
\end{proof}

We are now in position to study the solvability of our RBDSDEJ \eqref{backw} associated with parameters
$(f(.,\Theta ),g(.,\Theta ),\xi )$.
%
\begin{theorem}%
\label{existence}
Under the assumptions \emph{\textbf{(A1.1)}}--\emph{\textbf{(A1.4)}}, there exists $\beta_0>0$ such that for all $\beta\geq\beta_0$ the RBDSDEJ \eqref{backw} admits a unique solution
$(Y,Z,U,K,C)\in \mathscr{B}^{2}_{\beta }(\mathbf{R}^{k})\times
\mathscr{S}^{2}(\mathbf{R}^{k})\times \mathscr{S}^{2}(\mathbf{R}^{k})$.
\end{theorem}

\begin{proof}
\textbf{(i) Existence.} Our strategy in the proof of existence is to use the
Picard approximate sequence. To this end, we consider the sequence
$(\Theta ^{n})_{n\geq 0}:= (Y^{n},Z^{n},\break U^{n})_{n\geq 0}\in
\mathscr{B}^{2}_{\beta }(\mathbf{R}^{k})$ defined recursively by
$Y^{0} = Z^{0} = U^{0} = 0$ and for any $n \geq 1$,
$\tau \in \mathcal{T}_{[0,T]}$,
%
\begin{equation}
\label{eq9}
\left \{
\begin{array}{l}
Y^{n+1}_{\tau } = \xi _{T}+\displaystyle \int _{\tau }^{T} f\left (s,
\Theta ^{n}_{s}\right )ds+\int _{\tau }^{T} g\left (s,\Theta ^{n}_{s}
\right )dB_{s} -\int _{\tau }^{T}Z^{n+1}_{s} dW_{s}
\\
\hspace{1cm}
-\displaystyle \int _{\tau }^{T}\int _{E} U^{n+1}_{s}(e)
\widetilde{\mu }(ds,de)+K^{n+1}_{T}-K^{n+1}_{\tau }+C^{n+1}_{T-}-C^{n+1}_{
\tau -},
\\
Y^{n+1}_{\tau }\geq \xi _{\tau }\text{ a.s., }
\\
\displaystyle \int _{0}^{T}\mathbh{1}_{\{Y^{n+1}_{t}>\xi _{t}\}}dK^{c,n+1}_{t}=0
\text{ a.s.}\quad \text{and}\quad  (Y^{n+1}_{\tau -}-\xi _{\tau -})\Delta K^{d,n+1}_{\tau }=0 \text{ a.s.,}
\\
(Y^{n+1}_{\tau }-\xi _{\tau })\Delta C^{n+1}_{\tau }=0 \text{ a.s.}
\end{array}
\right .
\end{equation}
Since for
$n \geq 0,\,\, (Y^{n},Z^{n},U^{n})\in \mathscr{B}^{2}_{\beta }(
\mathbf{R}^{k})$, by virtue of Proposition \ref{pro2}, \break  RBDSDEJ \eqref{eq9} has a unique solution
$(Y^{n+1},Z^{n+1},U^{n+1},K^{n+1},C^{n+1})\in \mathscr{B}^{2}_{\beta }(
\mathbf{R}^{k})\times\break  \mathscr{S}^{2}(\mathbf{R}^{k})\times
\mathscr{S}^{2}(\mathbf{R}^{k})$.

In the sequel, we shall show that $(Y^{n}, Z^{n}, U^{n})_{n\geq 0}$ is
a Cauchy sequence in the Banach space
$\mathscr{B}^{2}_{\beta }(\mathbf{R}^{k})$. We define
$\overline{\Re }^{n+1} = \Re ^{n+1}- \Re ^{n}$ for
$\Re \in \{Y, Z, U,K,C\}$, and
\begin{equation*}
\forall h \in \left \{  f,g\right \}  , \quad \overline{h}^{n}_{\Theta }(t)
= h(t,\Theta ^{n}_{t}) - h(t,\Theta ^{n-1}_{t}), \quad t\le T.
\end{equation*}
We derive that for any $n\geq 1$ the process
$(\overline{Y}^{n+1}, \overline{Z}^{n+1}, \overline{U}^{n+1},
\overline{K}^{n+1}, \overline{C}^{n+1} )$ satisfies the following equation
\begin{eqnarray*}
\overline{Y}^{n+1}_{t} &=& \int _{t}^{T}\overline{f}_{\Theta }^{n}(s)ds
+ \int _{t}^{T}\overline{g}_{\Theta }^{n}(s)dB_{s} - \int _{t}^{T}
\overline{Z}^{n+1}_{s}dW_{s}
\\
&&-\int _{t}^{T}\int _{E}\overline{U}^{n+1}_{s}(e)\widetilde{\mu }(ds,de)+
\overline{K}^{n+1}_{T}-\overline{K}^{n+1}_{t}+\overline{C}^{n+1}_{T-}-
\overline{C}^{n+1}_{t-}.
\end{eqnarray*}
Applying the Lemma \ref{lem1}, we have
%
\begin{eqnarray}
\label{eq10}
&&e^{\beta A_{t}}|\overline{Y}^{n+1}_{t}|^{2}+\beta \int _{t}^{T} e^{
\beta A_{s}} a^{2}_{s}|\overline{Y}^{n+1}_{s}|^{2}ds+\int _{t}^{T}e^{
\beta A_{s}}|\overline{Z}^{n+1}_{s}|^{2}ds
\nonumber
\\
&=&2\int _{t}^{T}e^{\beta A_{s}}\langle \overline{Y}^{n+1}_{s-},
\overline{f}^{n}_{\Theta }(s)\rangle ds+2\int _{t}^{T}e^{\beta A_{s}}
\langle \overline{Y}^{n+1}_{s-},d\overline{K}^{n+1}_{s}\rangle
\nonumber
\\
&&-2\int _{t}^{T}e^{\beta A_{s}}\langle \overline{Y}^{n+1}_{s-},
\overline{Z}^{n+1}_{s}dW_{s}\rangle -2\int _{t}^{T}\int _{E} e^{
\beta A_{s}}\langle \overline{Y}^{n+1}_{s-}, \overline{U}^{n+1}_{s}(e)
\widetilde{\mu }(ds,de)\rangle
\nonumber
\\
&&%
\hspace{-1cm}%
+2\int _{t}^{T}e^{\beta A_{s}}\langle \overline{Y}^{n+1}_{s-},
\overline{g}^{n}_{\Theta }(s)dB_{s}\rangle +\int _{t}^{T}e^{\beta A_{s}}|
\overline{g}^{n}_{\Theta }(s)|^{2}ds+2\int _{t}^{T}e^{\beta A_{s}}
\langle \overline{Y}^{n+1}_{s},d\overline{C}^{n+1}_{s}\rangle
\nonumber
\\
&&-\sum _{t<s\leq T}e^{\beta A_{s}}(\Delta \overline{Y}^{n+1}_{s})^{2}-
\sum _{t\leq s<T}e^{\beta A_{s}}(\Delta _{+} \overline{Y}^{n+1}_{s})^{2}.
\end{eqnarray}
From Remark \ref{rem02}, the processes $\overline{K}^{n+1}$ and
$\mu $ do not have jumps in common, but $\overline{K}^{n+1}$ jumps at
predictable stopping times\index{predictable stopping times} and $\mu $ jumps only at totally inaccessible
stopping times,\index{totally inaccessible stopping times} then
\begin{equation*}
\int _{t}^{T}e^{\beta A_{s}}\|\overline{U}^{n+1}_{s}\|^{2}_{\lambda }ds-
\sum _{t<s\leq T}e^{\beta A_{s}}(\Delta \overline{Y}^{n+1}_{s})^{2}
\leq -\int _{t}^{T}\int _{E}e^{\beta A_{s}}|\overline{U}^{n+1}_{s}(e)|^{2}
\widetilde{\mu }(ds,de).
\end{equation*}
On the other hand, by using the Skorokhod\index{Skorokhod condition} and minimality conditions\index{minimality condition} on
$\overline{K}^{n+1}$ and $\overline{C}^{n+1}$ we can show that
$\langle \overline{Y}^{n+1}_{s-},d\overline{K}^{n+1}_{s}\rangle
\leq 0$ and
$\langle \overline{Y}^{n+1}_{s},d\overline{C}^{n+1}_{s}\rangle
\leq 0$. Moreover, from the assumptions \textbf{(A1.1)}--\textbf{(A1.2)}, we deduce
that for any $\varepsilon >0$,
\begin{eqnarray*}
2 \langle \overline{Y}_{s}^{n+1}, \overline{f}_{\Theta }^{n}(s)
\rangle &\leq & 2|\overline{Y}_{s}^{n+1}|\left (\gamma _{s}|
\overline{Y}_{s}^{n}| + \kappa _{s}|\overline{Z}_{s}^{n}| + \sigma _{s}
\|\overline{U}_{s}^{n}\|_{\lambda }\right )
\\
&&%
\hspace{-1.5cm}%
\;\;\leq \;\; \left (\gamma _{s} + \frac{1}{\varepsilon }[\kappa ^{2}_{s}
+ \sigma ^{2}_{s}]\right )|\overline{Y}_{s}^{n+1}|^{2} + \gamma _{s}|
\overline{Y}_{s}^{n}|^{2} + \varepsilon \left (|\overline{Z}_{s}^{n}|^{2}
+ \|\overline{U}_{s}^{n}\|_{\lambda }^{2}\right )
\\
&&%
\hspace{-1.5cm}%
\;\;\leq \;\; \left (1 + \frac{1}{\varepsilon }\right )a^{2}_{s}|
\overline{Y}_{s}^{n+1}|^{2} + a^{2}_{s}|\overline{Y}_{s}^{n}|^{2} +
\varepsilon \left (|\overline{Z}_{s}^{n}|^{2} + \|\overline{U}_{s}^{n}
\|_{\lambda }^{2}\right )
\end{eqnarray*}
and
\begin{equation*}
|\overline{g}_{\Theta }^{n}(s)|^{2}\leq a^{2}_{s}|\overline{Y}_{s}^{n}|^{2}
+ \alpha \left (|\overline{Z}_{s}^{n}|^{2} + \|\overline{U}_{s}^{n}
\|_{\lambda }^{2}\right ).
\end{equation*}
Plugging these inequalities in \eqref{eq10}, and taking the expectation
in both side, we deduce that, for any $\beta >0$ and $\varepsilon >0$,
\begin{eqnarray*}
&&\left (\beta - 1 - \frac{1}{\varepsilon }\right ){\mathbf{E}}\int _{t}^{T}e^{
\beta A_{s}}a^{2}_{s}|\overline{Y}_{s}^{n+1}|^{2}ds + {\mathbf{E}}\int _{t}^{T}e^{
\beta A_{s}}|\overline{Z}_{s}^{n+1}|^{2}ds
\\
&&+ {\mathbf{E}}\int _{t}^{T}e^{\beta A_{s}}\|\overline{U}^{n+1}_{s}\|^{2}_{\lambda }ds
\\
&\leq & 2{\mathbf{E}}\int _{t}^{T}e^{\beta A_{s}}a^{2}_{s}|\overline{Y}_{s}^{n}|^{2}ds
\\
&& + (\varepsilon + \alpha )\left ({\mathbf{E}}\int _{t}^{T}e^{\beta A_{s}}|
\overline{Z}_{s}^{n}|^{2}ds +{\mathbf{E}}\int _{t}^{T}e^{\beta A_{s}}\|
\overline{U}^{n}_{s}\|^{2}_{\lambda }ds\right ).
\end{eqnarray*}
Fix $\varepsilon >0$ and define
$\overline{c} = 2/(\varepsilon + \alpha )$ and
$\beta _{0} = 1 + \overline{c}+ 1/\varepsilon $. Choosing
$\beta \geq \beta _{0}$, we obtain
\begin{eqnarray*}
&&{\mathbf{E}}\bigg [\overline{c}\int _{t}^{T}e^{\beta A_{s}}a^{2}_{s}|
\overline{Y}_{s}^{n+1}|^{2}ds+ \int _{t}^{T}e^{\beta A_{s}}|
\overline{Z}_{s}^{n+1}|^{2}ds + \int _{t}^{T}e^{\beta A_{s}}\|
\overline{U}^{n+1}_{s}\|^{2}_{\lambda }ds\bigg ]
\\
&\leq & (\varepsilon + \alpha ){\mathbf{E}}\bigg [\overline{c}\int _{t}^{T}e^{
\beta A_{s}}a^{2}_{s}|\overline{Y}_{s}^{n}|^{2}ds + \int _{t}^{T}e^{
\beta A_{s}}|\overline{Z}_{s}^{n}|^{2}ds+ \int _{t}^{T}e^{\beta A_{s}}
\|\overline{U}^{n}_{s}\|^{2}_{\lambda }ds\bigg ]
\end{eqnarray*}
and by iterations we deduce that
\begin{eqnarray*}
&&\overline{c} \left \|  \overline{Y}^{n+1}\right \|  ^{2}_{\mathscr{M}_{\beta }^{2,a}(\mathbf{R}^{k})} + \left \|  \overline{Z}^{n+1}\right
\|  ^{2}_{\mathscr{M}^{2}_{\beta }(\mathbf{R}^{k})} + \left \|
\overline{U}^{n+1}\right \|  ^{2}_{\mathscr{L}^{2}_{\beta }(\mathbf{R}^{k})}
\\
&\le & (\varepsilon + \alpha )^{n} \left (\overline{c} \left \|
\overline{Y}^{1}\right \|  ^{2}_{\mathscr{M}^{2,a}_{\beta }(\mathbf{R}^{k})}
+ \left \|  \overline{Z}^{1}\right \|  ^{2}_{\mathscr{M}^{2}_{\beta }(
\mathbf{R}^{k})} + \left \|  \overline{U}^{1}\right \|  ^{2}_{
\mathscr{L}^{2}_{\beta }(\mathbf{R}^{k})}\right ).
\end{eqnarray*}
Hence, choosing $\varepsilon > 0$ such that
$\varepsilon + \alpha < 1$, we deduce that
$(Y^{n}, Z^{n}, U^{n})_{n \geq 1}$ is a Cauchy sequence in the Banach space
$\mathcal{A}^{2}_{\beta }(\mathbf{R}^{k})$. It remains to show that
$(Y^{n})_{n \geq 1}$ is a Cauchy sequence in
${\mathscr{S}}^{2}_{\beta }(\mathbf{R}^{k})$. To this end, we define for
any integers $n, m \geq 1$ $\Re ^{n,m} = \Re ^{n}-\Re ^{m}$ for
$\Re \in \{Y, Z, U, K,C\}$, and
\begin{equation*}
\forall h \in \left \{  f,\,g\right \}  ,\,\quad h^{n,m}_{\Theta }(t) = h(t,
\Theta ^{n}_{t}) - h(t,\Theta ^{m}_{t}), \quad t \leq T.
\end{equation*}
Then it is readily seen that
%
\begin{eqnarray}
\label{eqq}
Y_{t}^{n+1, m+1} &=& \int _{t}^{T}f_{\Theta }^{n, m}(s)ds + \int _{t}^{T}g_{\Theta }^{n, m}(s)dB_{s}-\int _{t}^{T}Z_{s}^{n+1, m+1}dW_{s}
\nonumber
\\
&&-\int _{t}^{T}\int _{E}U_{s}^{n+1, m+1}(e)\widetilde{\mu }(ds,de)+
K^{n+1,m+1}_{T}-K^{n+1,m+1}_{t}
\nonumber
\\
&&+C^{n+1,m+1}_{T-}-C^{n+1,m+1}_{t-}.
\end{eqnarray}
Applying Lemma \ref{lem1} to \eqref{eqq}, and taking the essential supremum
over $\tau \in \mathcal{T}_{[0,T]}$ and then the expectation on both sides
we get
\begin{eqnarray*}
&&{\mathbf{E}}\left (\operatorname*{ess
\,sup}_{\nu \in \mathcal{T}_{[0,T]}}e^{\beta A_{\tau }}|Y^{n+1,m+1}_{\tau }|^{2}\right )+\beta {\mathbf{E}}\int _{t}^{T}e^{\beta A_{s}}a^{2}_{s}|Y_{s}^{n+1,m+1}|^{2}ds
\\
&&+{\mathbf{E}}\int _{t}^{T}e^{\beta A_{s}}|Z_{s}^{n+1,m+1}|^{2}ds+{\mathbf{E}}
\int _{t}^{T}e^{\beta A_{s}}\|U_{s}^{n+1,m+1}\|^{2}_{\lambda }ds
\\
&\leq & 2{\mathbf{E}}\int _{t}^{T}e^{\beta A_{s}}\langle Y_{s-}^{n+1,m+1},f_{\Theta }^{n,m}(s)\rangle ds + {\mathbf{E}}\int _{t}^{T}e^{\beta A_{s}}| g_{\Theta }^{n,m}(s)|^{2}ds
\\
&&+2{\mathbf{E}}\operatorname*{ess
\,sup}_{\nu \in \mathcal{T}_{[0,T]}}\left |\int _{0}^{\tau }e^{\beta A_{s}}
\langle Y_{s-}^{n+1,m+1},Z_{s}^{n+1,m+1}dW_{s}\rangle \right |
\\
&&+2{\mathbf{E}}\operatorname*{ess
\,sup}_{\nu \in \mathcal{T}_{[0,T]}}\left |\int _{0}^{\tau }e^{\beta A_{s}}
\langle Y_{s-}^{n+1,m+1},g_{\Theta }^{n,m}(s)dB_{s}\rangle \right |
\\
&&+2{\mathbf{E}}\operatorname*{ess
\,sup}_{\nu \in \mathcal{T}_{[0,T]}}\left |\int _{0}^{\tau }\int _{E} e^{
\beta A_{s}}\langle Y_{s-}^{n+1,m+1}, U_{s}^{n+1,m+1}(e)
\widetilde{\mu }(ds ,de)\rangle \right |.
\end{eqnarray*}
But, for any $\varepsilon >0$,
\begin{eqnarray*}
2 \langle Y_{s}^{n+1,m+1},f_{\Theta }^{n,m}(s)\rangle \leq
\frac{1}{\varepsilon }a^{2}_{s}|Y_{s}^{n+1,m+1}|^{2}+\varepsilon
\left |\frac{f_{\Theta }^{n,m}(s)}{a_{s}}\right |^{2}.
\end{eqnarray*}
Moreover, by the Burkh\"{o}lder--Davis--Gundy inequality, there exists a universal
constant $c$ such that
\begin{eqnarray*}
&&%
\hspace{-1cm}%
2{\mathbf{E}}\operatorname*{ess
\,sup}_{\tau \in \mathcal{T}_{[0,T]}}\left |\int _{0}^{\tau }e^{\beta A_{s}}
\langle Y_{s-}^{n+1,m+1},Z_{s}^{n+1,m+1}dW_{s}\rangle \right |
\\
&&%
\hspace{2cm}%
\;\;\leq \;\;\frac{1}{4}\|Y^{n+1,m+1}\|_{\mathscr{S}^{2}_{\beta }(
\mathbf{R}^{k})}^{2}+4c^{2}\|Z^{n+1,m+1}\|_{\mathscr{M}^{2}_{\beta }(
\mathbf{R}^{k\times d})}^{2},
\end{eqnarray*}
\begin{eqnarray*}
&&%
\hspace{-1cm}%
2{\mathbf{E}}\operatorname*{ess
\,sup}_{\tau \in \mathcal{T}_{[0,T]}}\left |\int _{0}^{\tau }\int _{E} e^{
\beta A_{s}}\langle Y_{s-}^{n+1,m+1}, U_{s}^{n+1,m+1}(e)
\widetilde{\mu }(ds ,de)\rangle \right |
\\
&&%
\hspace{2cm}%
\;\;\leq \;\; \frac{1}{4}\|Y^{n+1,m+1}\|_{\mathscr{S}^{2}_{\beta }(
\mathbf{R}^{k})}^{2}+4c^{2}\|U^{n+1,m+1}\|_{\mathscr{L}^{2}_{\beta }(
\mathbf{R}^{k})}^{2}
\end{eqnarray*}
and
\begin{eqnarray*}
&&%
\hspace{-1cm}%
2{\mathbf{E}}\operatorname*{ess
\,sup}_{\tau \in \mathcal{T}_{[0,T]}}\left |\int _{0}^{\tau }e^{\beta A_{s}}
\langle Y_{s-}^{n+1,m+1},g_{\Theta }^{n,m}(s)dB_{s}\rangle \right |
\\
&&%
\hspace{2cm}%
\;\;\leq \;\;\frac{1}{4}\|Y^{n+1,m+1}\|_{\mathscr{S}^{2}_{\beta }(
\mathbf{R}^{k})}^{2}+4c^{2}\|g^{n,m}_{\Theta }\|_{\mathscr{M}^{2}_{\beta }(\mathbf{R}^{k\times \ell })}^{2}.
\end{eqnarray*}
Hence, there exists $\mathcal{C}>0$ such that
\begin{eqnarray*}
&&{\mathbf{E}}\left (\operatorname*{ess
\,sup}_{\nu \in \mathcal{T}_{[0,T]}}e^{\beta A_{\tau }}|Y^{n+1,m+1}_{\tau }|^{2}\right )
\\
&\leq & \mathcal{C}\left ({\mathbf{E}}\int _{0}^{T}e^{\beta A_{s}}\left |
\frac{f_{\Theta }^{n,m}(s)}{a_{s}}\right |^{2} ds + {\mathbf{E}}\int _{0}^{T}e^{
\beta A_{s}}| g_{\Theta }^{n,m}(s)|^{2}ds\right )
\\
&\leq &\mathcal{C}\left (4{\mathbf{E}}\int _{0}^{T}e^{\beta A_{s}}a^{2}_{s}|Y_{s}^{n,m}|^{2}ds
\right .
\\
&&\left . + (3 + \alpha )\left ({\mathbf{E}}\int _{0}^{T}e^{\beta A_{s}}|Z_{s}^{n,m}|^{2}ds
+{\mathbf{E}}\int _{0}^{T}e^{\beta A_{s}}\|U^{n,m}_{s}\|^{2}_{\lambda }ds
\right )\right ).
\end{eqnarray*}
Since $(Y^{n},Z^{n},U^{n})_{n \geq 1}$ is a Cauchy sequence in
$\mathcal{A}^{2}_{\beta }(\mathbf{R}^{k})$, we deduce that
$(Y^{n})_{n \geq 1}$ is a Cauchy sequence in
${\mathscr{S}}^{2}_{\beta }(\mathbf{R}^{k})$. Hence,
$(Y^{n}, Z^{n}, U^{n})_{n \geq 1}$ is a Cauchy sequence in the Banach space
$\mathscr{B}^{2}_{\beta }(\mathbf{R}^{k})$, so it converges in
$\mathscr{B}^{2}_{\beta }(\mathbf{R}^{k})$ to a limit
$\Theta =(Y,Z,U)$. Now let us show that $(Y, Z, U)$, with the additional
Mertens process $(K,C)$,\index{Mertens process} is a solution to RBDSDEJ \eqref{backw}.\vadjust{\goodbreak}%

Since $(Y^{n}, Z^{n}, U^{n})_{n \geq 1}$ converges in
$\mathscr{B}^{2}_{\beta }(\mathbf{R}^{k})$ to a limit $(Y, Z, U)$, we
have
%
\begin{align}
\label{eq11}
\displaystyle \lim _{n\rightarrow + \infty } \left \|  (Y^{n}-Y, Z^{n}-Z,
U^{n}-U)\right \|^2  _{\mathscr{B}^{2}_{\beta }(\mathbf{R}^{k})}=0.
\end{align}
Using the Cauchy--Schwarz inequality and \eqref{eq11}, we deduce from \textbf{(A1.1)}
and \textbf{(A1.2)}
\begin{align*}
{\mathbf{E}}\bigg (\bigg |\int _{t}^{T}(f(s,&\Theta _{s}^{n}) - f(s,\Theta _{s}))ds
\bigg |^{2}\bigg )
\\
&\leq \;\; {\mathbf{E}}\bigg (\frac{1}{\beta }\int _{t}^{T}e^{\beta A_{s}}
\frac{|f(s,Y_{s}^{n}, Z_{s}^{n}, U_{s}^{n}) - f(s,Y_{s}, Z_{s}, U_{s})|^{2}}{a^{2}_{s}}ds
\bigg )
\\
&\leq \;\; \frac{3}{\beta }{\mathbf{E}}\bigg ( \int _{t}^{T}e^{\beta A_{s}}a^{2}_{s}|Y_{s}^{n}-Y_{s}|^{2}ds
+ \int _{t}^{T}e^{\beta A_{s}}|Z_{s}^{n}-Z_{s}|^{2}ds
\\
&\;\;+\int _{t}^{T}e^{\beta A_{s}}\|U_{s}^{n}-U_{s}\|^{2}_{\lambda }ds
\bigg )\xrightarrow[n\to +\infty ]{}0.
\end{align*}
Similarly, by the Burkh\"{o}lder--Davis--Gundy inequality and \eqref{eq11}, we have
\begin{align*}
{\mathbf{E}}\bigg (\sup _{0\leq t \leq T}\bigg |\int _{t}^{T}&g(s,\Theta _{s}^{n})dB_{s}
- \int _{t}^{T}g(s,\Theta _{s})dB_{s}\bigg |^{2}\bigg )
\\
&\leq \;\; {\mathbf{E}}\left (\int _{t}^{T}|g(s,Y_{s}^{n}, Z_{s}^{n}, U_{s}^{n})
- g(s,Y_{s}, Z_{s}, U_{s})|^{2}ds\right )
\\
&\leq \;\; {\mathbf{E}}\bigg (\int _{t}^{T}e^{\beta A_{s}}a^{2}_{s}|Y_{s}^{n}-Y_{s}|^{2}ds
+ \alpha \int _{t}^{T}e^{\beta A_{s}}|Z_{s}^{n}-Z_{s}|^{2}ds
\\
&\;\; +\alpha \int _{t}^{T}e^{\beta A_{s}}\|U_{s}^{n}-U_{s}\|^{2}_{\lambda }ds\bigg )\xrightarrow[n\to +\infty ]{}0.
\end{align*}
Moreover, since $A_{s}\geq 0$ for all $s\le T$, we have
\begin{equation*}
{\mathbf{E}}\left (\sup _{0\leq t \leq T}\left |\int _{t}^{T} Z_{s}^{n}dW_{s}
- \int _{t}^{T} Z_{s}dW_{s}\right |^{2}\right )\leq {\mathbf{E}}\int _{t}^{T}e^{
\beta A_{s}}|Z_{s}^{n}-Z_{s}|^{2}ds\xrightarrow[n\to +\infty ]{}0
\end{equation*}
and
\begin{align*}
{\mathbf{E}}\bigg (\sup _{0\leq t \leq T}\bigg |\int _{t}^{T}\int _{E} U_{s}^{n}(e)
\widetilde{\mu }(de,ds) &- \int _{t}^{T}\int _{E} U_{s}(e)
\widetilde{\mu }(de,ds)\bigg |^{2}\bigg )
\\
&\leq \;\; {\mathbf{E}}\int _{t}^{T}e^{\beta A_{s}}\|U_{s}^{n}-U_{s}\|^{2}_{\lambda }ds \xrightarrow[n\to +\infty ]{}0.
\end{align*}
For each $\tau \in \mathcal{T}_{[0,T]}$, let
\begin{eqnarray*}
\widetilde{K}_{\tau }=K_{\tau }-C_{\tau -}&=&Y_{0}-Y_{\tau }-\displaystyle
\int _{0}^{\tau }f\left (s,\Theta _{s}\right )ds-\int _{0}^{\tau }g\left (s,
\Theta _{s}\right )dB_{s}
\\
&&+\int _{0}^{\tau }Z_{s} dW_{s}+\displaystyle \int _{0}^{\tau }\int _{E} U_{s}(e)
\widetilde{\mu }(ds,de).
\end{eqnarray*}
Then, we can easily show that
$\|\widetilde{K}^{n}-\widetilde{K}\|^{2}_{\mathscr{S}^{2}}
\longrightarrow 0,\,\,\text{as}\,\,n\longrightarrow +\infty $. So, letting
${n \longrightarrow + \infty }$ in \eqref{eq9}, we deduce that
$(Y,Z,U,K,C)$ is a solution to RBDSDEJ \eqref{backw}.

\textbf{(ii) Uniqueness.} Let $(Y^{1},Z^{1},U^{1},K^{1},C^{1})$ and
$(Y^{2},Z^{2},U^{2},K^{2},C^{2})$ be two solutions to RBDSDEJ \eqref{backw}. We define $\overline{\Re }= \Re ^{1}-\Re ^{2}$ for
$\Re \in \{Y,Z,U,K,C\}$ and
\begin{equation*}
\forall h \in \left \{  f,g\right \}  , \quad \overline{h}_{\Theta }(t) = h(t,
\Theta ^{1}_{t}) - h(t, \Theta ^{2}_{t}), \quad t \leq T.
\end{equation*}
Thus the process
$(\overline{Y}, \overline{Z}, \overline{U}, \overline{K},
\overline{C})$ satisfies the following equation
%
\begin{eqnarray}
\label{eq12}
\overline{Y}_{t}&=& \int _{t}^{T} \overline{f}_{\Theta }(s)ds + \int _{t}^{T}
\overline{g}_{\Theta }(s)dB_{s} - \int _{t}^{T}\overline{Z}_{s}dW_{s} +
\int _{t}^{T}\int _{E}\overline{U}_{s}(e)\widetilde{\mu }(ds,de)
\nonumber
\\
&&+\overline{K}_{T}-\overline{K}_{t}+\overline{C}_{T-}-\overline{C}_{t-}.
\end{eqnarray}
Applying Lemma \ref{lem1} to \eqref{eq12} and taking into consideration
Remark \ref{rem02}, we have
\begin{align*}
{\mathbf{E}}\left [e^{\beta A_{t}}|\overline{Y}_{t}|^{2}\right ] &+
\beta {\mathbf{E}}\int _{t}^{T}e^{\beta A_{s}}a^{2}_{s}|\overline{Y}_{s}|^{2}ds
+ {\mathbf{E}}\int _{t}^{T}e^{\beta A_{s}}(|\overline{Z}_{s}|^{2}+\|
\overline{U}_{s}\|^{2}_{\lambda }) ds
\nonumber
\\
& \leq \;\; 2 {\mathbf{E}}\int _{t}^{T}e^{\beta A_{s}}\langle
\overline{Y}_{s},\overline{f}_{\Theta }(s)\rangle ds + {\mathbf{E}}\int _{t}^{T}e^{
\beta A_{s}}| \overline{g}_{\Theta }(s)|^{2}ds.
\end{align*}
By the same computations as before (by using the assumptions \textbf{(A1.1)}--\textbf{(A1.2)}),
we have, for any $\varepsilon > 0$,
\begin{equation*}
2\langle \overline{Y}_{s},\overline{f}_{\Theta }(s)\rangle \leq \left (2+
\frac{2}{\varepsilon }\right )a^{2}_{s}|\overline{Y}_{s}|^{2} +
\varepsilon (|\overline{Z}_{s}|^{2} + |\overline{U}_{s}|^{2}),
\end{equation*}
and
\begin{equation*}
| \overline{g}_{\Theta }(s)|^{2} \le a^{2}_{s}|\overline{Y}_{s}|^{2} +
\alpha (|\overline{Z}_{s}|^{2} + |\overline{U}_{s}|^{2}).
\end{equation*}
Hence, choosing $\varepsilon > 0,\; \beta > 0 $ such that
$\varepsilon + \alpha < 1$ and $\beta > 3 + 2/\varepsilon $, we deduce
that
\begin{align*}
{\mathbf{E}}\left [e^{\beta A_{t}}|\overline{Y}_{t}|^{2}\right ] &+
\left (\beta -3-\frac{2}{\varepsilon }\right ){\mathbf{E}}\int _{t}^{T}e^{
\beta A_{s}}a^{2}_{s}|\overline{Y}_{s}|^{2}ds
\nonumber
\\
& + (1-\varepsilon -\alpha ){\mathbf{E}}\left [\int _{t}^{T}e^{\beta A_{s}}|
\overline{Z}_{s}|^{2}ds + \int _{t}^{T}e^{\beta A_{s}}\|
\overline{U}_{s}\|^{2}_{\lambda }ds\right ]\leq 0.
\end{align*}
It follows that
$(\overline{Y},\overline{Z},\overline{U})=(0,0,0)$, and thus
$(\overline{K},\overline{C})=(0,0)$.
\end{proof}

\subsection{Comparison theorem}

In all what follows, we are interested in one-dimensional RBDSDEJs (i.e.
$k =1$). We consider the RBDSDEJs associated with parameters
$(f^{i}(.,\Theta ),g(.,\Theta ),\xi ^{i})$ for $i =1,2$ where
$\Theta ^{i}$ stands for the process $(Y^{i}, Z^{i}, U^{i})$. Let us state
the following assumption
\begin{equation*}
{\mbox{\textbf{(A1.6):}}}
\begin{cases}
\xi ^{1}_{t}\le \xi ^{2}_{t} \text{ a.s.}\quad \forall t\leq T
\\
f^{1}(t,y,z,u) \le f^{2}(t,y,z,u)\text{ a.s.}\quad \forall (t,y,z,u)\in [0,
T]\times \mathbf{R}\times \mathbf{R}^{d} \times \mathscr{L}_{\lambda }.
\end{cases}\vadjust{\eject}
\end{equation*}
Then we have the following comparison result.
%
\begin{theorem}%
\label{comp-Lipsch}
Let $(Y^{i},Z^{i},U^{i},K^{i},C^{i})$ be a solution to RBDSDEJs associated
with parameters $(f^{i}(.,\Theta ),g(.,\Theta ),\xi ^{i})$ for
$i =1,2$. Under the assumptions \emph{\textbf{(A1.1)}}--\emph{\textbf{(A1.4)}} and \emph{\textbf{(A1.6)}} we
have
\begin{equation*}
\forall \, t \le T, \quad Y^{1}_{t} \le Y^{2}_{t}, \quad {\mathbf{P}}\text{-}a.s.
\end{equation*}
\end{theorem}

\begin{proof}
Define $\widehat{\Re }=\Re ^{1}-\Re ^{2}$ for
$\Re \in \{Y,Z,U,K,C,\xi \}$. Then the process
$(\widehat{Y},\widehat{Z},\widehat{U},\break \widehat{K},\widehat{C})$ satisfies
the following equation
\begin{align*}
\widehat{Y}_{t} =\;\widehat{\xi }_{T} &+ \int _{t}^{T}[f^{1}(s,
\Theta _{s}^{1})-f^{2}(s,\Theta _{s}^{2})]ds + \int _{t}^{T}[g(s,
\Theta _{s}^{1})-g(s,\Theta _{s}^{2})]dB_{s}
\\
&-\int _{t}^{T}\widehat{Z}_{s}dW_{s} - \int _{t}^{T} \int _{E}
\widehat{U}_{s}(e)\widetilde{\mu }(ds,de)+\widehat{K}_{T}-\widehat{K}_{t}+
\widehat{C}_{T-}-\widehat{C}_{t-}.
\end{align*}
Applying Lemma \ref{lem1}, taking into account Remark \ref{rem02} and
taking the expectation, we obtain, for all $t\leq T$,
%
\begin{eqnarray}
\label{eqx}
&&{\mathbf{E}}\left [e^{\beta A_{t}}|\widehat{Y}_{t}^{+}|^{2}\right ] +
\beta {\mathbf{E}}\int _{t}^{T}\mathbh{1}_{\left \{  \widehat{Y}_{s} > 0
\right \}  }e^{\beta A_{s}}a^{2}_{s}|\widehat{Y}_{s}|^{2}ds
\nonumber
\\
&& +{\mathbf{E}}\int _{t}^{T} \mathbh{1}_{\left \{  \widehat{Y}_{s} > 0
\right \}  } e^{\beta A_{s}} |\widehat{Z}_{s}|^{2}ds + {\mathbf{E}}\int _{t}^{T}
\mathbh{1}_{\left \{  \widehat{Y}_{s}>0\right \}  }e^{\beta A_{s}}\|
\widehat{U}_{s}\|^{2}_{\lambda }ds
\nonumber
\\
&\leq & {\mathbf{E}}\left [e^{\beta A_{T}}|\widehat{\xi }_{T}^{+}|^{2}
\right ] + 2{\mathbf{E}}\int _{t}^{T} \mathbh{1}_{\left \{  \widehat{Y}_{s}>0
\right \}  } e^{\beta A_{s}}\widehat{Y}_{s}^{+}[f^{1}(s,\Theta _{s}^{1})-f^{2}(s,
\Theta _{s}^{2})]ds
\nonumber
\\
&&+ {\mathbf{E}}\int _{t}^{T} \mathbh{1}_{\left \{  \widehat{Y}_{s}>0
\right \}  }e^{\beta A_{s}}|g(s,\Theta _{s}^{1})-g(s,\Theta _{s}^{2})|^{2}ds.
\end{eqnarray}
By assumption \textbf{(A1.6)}, we have
${\mathbf{E}}[e^{\beta A_{T}} |\widehat{\xi }_{T}^{+}|]=0$ and
$\widehat{Y}_{s}^{+}[f^{1}(s,\Theta _{s}^{1})-f^{2}(s,\Theta _{s}^{1})]
\leq 0$, and due to the assumptions \textbf{(A1.1)}--\textbf{(A1.2)}, we get,
for any $\varepsilon >0$,
\begin{eqnarray*}
&&2{\mathbf{E}}\int _{t}^{T} \mathbh{1}_{\left \{  \widehat{Y}_{s}>0\right
\}  } e^{\beta A_{s}}\widehat{Y}_{s}^{+}[f^{1}(s,\Theta _{s}^{1}) -f^{2}(s,
\Theta _{s}^{2})]ds
\\
&\leq & 2{\mathbf{E}}\int _{t}^{T} \mathbh{1}_{\left \{  \widehat{Y}_{s}>0
\right \}  } e^{\beta A_{s}}\widehat{Y}_{s}^{+}[f^{2}(s,\Theta _{s}^{1})-f^{2}(s,
\Theta _{s}^{2})]ds
\\
&\leq & \left (2 + \frac{2}{\varepsilon }\right ){\mathbf{E}}\int _{t}^{T}
\mathbh{1}_{\left \{  \widehat{Y}_{s}>0\right \}  } e^{\beta A_{s}} a^{2}_{s}|
\widehat{Y}_{s}^{+}|^{2}ds + \varepsilon {\mathbf{E}}\int _{t}^{T}
\mathbh{1}_{\left \{  \widehat{Y}_{s} > 0\right \}  } e^{\beta A_{s}} |
\widehat{Z}_{s}|^{2}ds
\\
&&+ \varepsilon {\mathbf{E}}\int _{t}^{T} \mathbh{1}_{\left \{  \widehat{Y}_{s}>0
\right \}  } e^{\beta A_{s}} \|\widehat{U}_{s}\|^{2}_{\lambda }ds
\end{eqnarray*}
and
\begin{eqnarray*}
&&{\mathbf{E}}\int _{t}^{T}\mathbh{1}_{\left \{  \widehat{Y}_{s}>0\right \}  }
e^{\beta A_{s}} |g(s,\Theta _{s}^{1})- g(s,\Theta _{s}^{2})|^{2}ds
\\
&\leq & {\mathbf{E}}\int _{t}^{T}\mathbh{1}_{\left \{  \widehat{Y}_{s}>0
\right \}  } e^{\beta A_{s}} a^{2}_{s} |\widehat{Y}_{s}|^{2}ds +
\alpha {\mathbf{E}}\int _{t}^{T}\mathbh{1}_{\left \{  \widehat{Y}_{s} > 0
\right \}  } e^{\beta A_{s}} (|\widehat{Z}_{s}|^{2}+\|\widehat{U}_{s}
\|_{\lambda }^{2})ds.
\end{eqnarray*}
Plugging these two last inequalities in \eqref{eqx}, we deduce that, for
any $\beta >0$ and $\varepsilon >0$,
\begin{eqnarray*}
&&{\mathbf{E}}\left [e^{\beta A_{t}}|\widehat{Y}_{t}^{+}|^{2}\right ] +
\beta {\mathbf{E}}\int _{t}^{T}\mathbh{1}_{\left \{  \widehat{Y}_{s} > 0
\right \}  }e^{\beta A_{s}}a^{2}_{s}|\widehat{Y}_{s}|^{2}ds
\nonumber
\\
&& +{\mathbf{E}}\int _{t}^{T} \mathbh{1}_{\left \{  \widehat{Y}_{s} > 0
\right \}  } e^{\beta A_{s}} |\widehat{Z}_{s}|^{2}ds + {\mathbf{E}}\int _{t}^{T}
\mathbh{1}_{\left \{  \widehat{Y}_{s}>0\right \}  }e^{\beta A_{s}}\|
\widehat{U}_{s}\|^{2}_{\lambda }ds
\nonumber
\\
&\leq & \left (3 + \frac{2}{\varepsilon }\right ){\mathbf{E}}\int _{t}^{T}
\mathbh{1}_{\left \{  \widehat{Y}_{s}>0\right \}  } e^{\beta A_{s}} a^{2}_{s}|
\widehat{Y}_{s}^{+}|^{2}ds
\nonumber
\\
&&+ (\varepsilon + \alpha )\left ({\mathbf{E}}\int _{t}^{T} \mathbh{1}_{
\left \{  \widehat{Y}_{s}>0\right \}  } e^{\beta A_{s}} |\widehat{Z}_{s}|^{2}ds
+ {\mathbf{E}}\int _{t}^{T}\mathbh{1}_{\left \{  \widehat{Y}_{s}>0\right \}  }
e^{\beta A_{s}} \|\widehat{U}_{s}\|^{2}_{\lambda }ds\right ).
\end{eqnarray*}
Choosing $\varepsilon =(1-\alpha )/2$ and taking
$\beta > 3 + 2/\varepsilon $, we derive that
\begin{equation*}
|\widehat{Y}_{t}^{+}|^{2} = 0 \quad \text{a.s.} \quad\forall t\leq T, \quad\text{i.e.,}\quad Y_{t}^{1}
\leq Y_{t}^{2}\quad \text{a.s.}\quad \forall t\leq T.\qedhere
\end{equation*}
\end{proof}

\section{Reflected BDSDEJs with stochastic growth condition}
\label{s3}

In this section we are interested in weakening the conditions on the coefficient
$f$. We are also interested in one-dimensional RBDSDEJs (i.e. $k =1$).
Let us state the new working assumptions.

\subsection{Assumptions}

We assume that the data $(f,g,\xi )$ satisfy the following assumptions
\textbf{(A2)}:
\begin{description}
\item[\bf(A2.1):] There exist four non-negative
${\mathscr{F}}^{W}_{t}$-measurable processes
$(\gamma _{t})_{t\leq T}$, $(\kappa _{t})_{t\leq T}$,
$(\sigma _{t})_{t\leq T}$ and $(\varrho _{t})_{t\leq T}$ such that the
condition \textbf{(A1.2)} holds, and there exists another
${\mathcal{F}}_{t}$-progressively measurable nonnegative process
$(\zeta _{t})_{t\le T}$ such that
$\displaystyle \frac{\zeta }{a}\in \mathscr{M}^{2}_{\beta }(\mathbf{R})$
and for all
$(t,y,z,u)\in [0, T]\times \mathbf{R}\times \mathbf{R}^{d}\times
\mathscr{L}_{\lambda }$,
\begin{equation*}
\vert f(t,y,z,u)\vert \le \zeta _{t}+ \gamma _{t} |y| + \kappa _{t} |z|
+ \sigma _{t}\|u\|_{\lambda }.
\end{equation*}
\item[\bf(A2.2):]
$f(\omega , t, \cdot , \cdot , \cdot ) : \mathbf{R}\times
\mathbf{R}^{d} \times \mathscr{L}_{\lambda }\to \mathbf{R}$ is continuous.
\item[\bf(A2.3):] The coefficient $g$ satisfies \textbf{(A1.1)}
for $\alpha \in\ ]0,1/2[\,$.
\item[\bf(A2.4):] The irregular barrier\index{irregular barrier}
$(\xi _{t})_{t\leq T}$ satisfies \textbf{(A1.4)}.
\end{description}

\subsection{Existence of a minimal solution\index{minimal solution}}

In this section, we will prove the existence of a minimal solution\index{minimal solution} to
RBDSDEJ \eqref{backw} under the conditions \textbf{(A2)}. First let us define
a minimal solution\index{minimal solution} 
as follows.
%
\begin{definition}
A solution $(Y,Z,U,K,C)$ to RBDSDEJ \eqref{backw} is called a minimal solution\index{minimal solution}
if for any other solution
$(Y^{\ast },Z^{\ast },U^{\ast },K^{\ast },C^{\ast })$ to \eqref{backw} we have, for
each $t \le T$,  $Y_{t} \le Y^{\ast }_{t}$.
\end{definition}

For fixed $(\omega ,t)$ in $\Omega \times [0,T]$, we define the sequence
$f_{n}(t,y,z,u)$ associated to the coefficient $f$ as follows: for all
$(y,y^{\prime })\in \mathbf{R}^{2}$,
$(z,z^{\prime })\in \mathbf{R}^{d}\times \mathbf{R}^{d}$ and
$(u,u^{\prime })\in \mathscr{L}_{\lambda }\times \mathscr{L}_{\lambda }$,
\begin{equation*}
f_{n}(t, y, z, u) = \inf _{y^{\prime }, z^{\prime }, u^{\prime }} [f(t, y^{\prime }, z^{\prime }, u^{\prime })+ n (\gamma _{t}|y-y^{\prime }|+ \kappa _{t}|z-z^{\prime }|+\sigma _{t}\|u-u^{\prime }\|_{\lambda })].
\end{equation*}
From Proposition 4.2 in \cite{sow-sagna}, the sequence $f_{n}$ is well
defined for each $n\ge 1$, and it satisfies:
\begin{itemize}
\item Linear growth condition: $\forall n\ge 1$,
%
\begin{equation}
\label{linear}
\forall (y, z, u)\in \mathbf{R}\times \mathbf{R}^{d}\times
\mathscr{L}_{\lambda }, \quad |f_{n}(t, y, z, u)| \le \zeta _{t} +
\gamma _{t}|y| + \kappa _{t} |z|+ \sigma _{t} \|u\|_{\lambda }.
\end{equation}
\item Monotonicity:
$\forall (y, z, u)\in \mathbf{R}\times \mathbf{R}^{d}\times
\mathscr{L}_{\lambda }$, $f_{n}(t, y, z, u)$ increases in $n$.
\item Convergence: If $ (y_{n}, z_{n}, u_{n}) \to (y, z, u)$ in
$\mathbf{R}\times \mathbf{R}^{d} \times \mathscr{L}_{\lambda }$ as
$n\to +\infty $, then
%
\begin{equation}
\label{convergence}
f_{n}(t, y_{n}, z_{n}, u_{n}) \; \xrightarrow[n\to +\infty ]{} f(t, y,
z, u).
\end{equation}
\item Lipschitz condition: $\forall n\ge 1$, and for all
$(y,y^{\prime })\in \mathbf{R}^{2}$,
$(z,z^{\prime })\in \mathbf{R}^{d}\times \mathbf{R}^{d}$ and
$(u,u^{\prime })\in \mathscr{L}_{\lambda }\times \mathscr{L}_{\lambda }$, we
have
\begin{equation*}
|f_{n}(t, y, z, u) - f_{n}(t, y^{\prime }, z^{\prime }, u^{\prime })| \le n
\gamma _{t}|y - y^{\prime }| + n\kappa _{t} |z - z^{\prime }| + n \sigma _{t}
\|u - u^{\prime }\|_{\lambda }.
\end{equation*}
\end{itemize}
We also define the function
\begin{equation*}
F(t,y,z,u)= \zeta _{t}+\gamma _{t} |y|+\kappa _{t}|z|+\sigma _{t}\|u
\|_{\lambda }.
\end{equation*}
Now, from Theorem \ref{existence}, there exist two processes
$\overline{\Theta }:=(\overline{Y}, \overline{Z}, \overline{U})$ and
$\Theta ^{n}:=(Y^{n}, Z^{n}, U^{n})$ which are the solutions to RBDSDEJs
associated with parameters{\break}
$(F(.,\overline{\Theta }),g(.,\overline{\Theta }),\xi )$ and
$(f_{n}(.,\Theta ^{n}),g(.,\Theta ^{n}),\xi )$, respectively.

From the definitions of $f_{n}$ and $F$ together with \eqref{linear}, we
observe that $\forall n\geq 1$, $f_{n}\leq f_{n+1}\leq F$. Then, due
to Theorem \ref{comp-Lipsch} we have
%
\begin{equation}
\label{estim-Yn}
\forall t \leq T, \quad Y_{t}^{1}\le Y_{t}^{n}\le Y_{t}^{n+1}\le
\overline{Y}_{t}.
\end{equation}

The proof of the main result of this section is based on the two next lemmas.
%
\begin{lemma}%
\label{borne}
Under the assumption \textbf{(A2)}, there exists a positive constant
$\Lambda $ depending on $\beta $ such that
\begin{equation*}
\left \Vert ~(\overline{Y}, \overline{Z}, \overline{U}) \right \Vert _{
\mathscr{B}^{2}_{\beta }(\mathbf{R})}^{2} \le \Lambda \left (\|\xi \|_{
\mathscr{S}^{2}_{2\beta }(\mathbf{R})}^{2} + \left \|
\frac{\zeta }{a}\right \|  ^{2}_{\mathscr{M}^{2}_{\beta }(\mathbf{R})} +
\|g(., 0)\|^{2}_{\mathscr{M}^{2}_{\beta }(\mathbf{R}^{\ell })}\right )
\end{equation*}
and for each $n\geq 1$
\begin{equation*}
\left \Vert (Y^{n}, Z^{n}, U^{n}) \right \Vert _{ \mathscr{B}^{2}_{\beta }(\mathbf{R})}^{2} \le \Lambda \left (\|\xi \|_{\mathscr{S}^{2}_{2
\beta }(\mathbf{R})}^{2} + \left \|  \frac{\zeta }{a}\right \|  ^{2}_{
\mathscr{M}^{2}_{\beta }(\mathbf{R})} + \|g(., 0)\|^{2}_{\mathscr{M}^{2}_{\beta }(\mathbf{R}^{\ell })}\right ).
\end{equation*}
\end{lemma}
\begin{proof}
We know that
%
\begin{eqnarray}
\label{eqqn}
\overline{Y}_{t} &=&\xi _{T} + \int _{t}^{T}F(s,\overline{\Theta }_{s})ds
+ \int _{t}^{T}g(s,\overline{\Theta }_{s})dB_{s}-\int _{t}^{T}
\overline{Z}_{s}dW_{s}
\nonumber
\\
&&- \int _{t}^{T} \int _{E}\overline{U}_{s}(e)\widetilde{\mu }(ds,de)+
\overline{K}_{T}-\overline{K}_{t}+\overline{C}_{T-}-\overline{C}_{t-},
\end{eqnarray}
where $(\overline{K},\overline{C})$ satisfies the Skorokhod\index{Skorokhod condition} and minimality
conditions.\index{minimality condition} Then, applying Lemma \ref{lem1} together with Remark \ref{rem02}, we deduce
%
\begin{eqnarray}
\label{Ito-Yt}
&&e^{\beta A_{t}}|\overline{Y}_{t}|^{2} + \beta \int _{t}^{T}e^{\beta A_{s}}a^{2}_{s}|
\overline{Y}_{s}|^{2}ds + \int _{t}^{T}e^{\beta A_{s}}|\overline{Z}_{s}|^{2}ds
+ \int _{t}^{T}e^{\beta A_{s}}\|\overline{U}_{s}\|^{2}_{\lambda }ds
\nonumber
\\
&\leq & e^{\beta A_{T}}|\xi |^{2} + 2\int _{t}^{T} e^{\beta A_{s}}
\overline{Y}_{s} F(s,\overline{\Theta }_{s})ds + 2 \int _{t}^{T} e^{
\beta A_{s}}\overline{Y}_{s-} g(s, \overline{\Theta }_{s})dB_{s}
\nonumber
\\
&& -2\int _{t}^{T} e^{\beta A_{s}} \overline{Y}_{s-} \overline{Z}_{s} dW_{s}
- 2\int _{t}^{T}\int _{E} e^{\beta A_{s}} \overline{Y}_{s-}
\overline{U}_{s}(e)\widetilde{\mu }(ds, de)
\nonumber
\\
&&+ \int _{t}^{T} e^{\beta A_{s}}|g(s,\overline{\Theta }_{s})|^{2}ds+ 2
\int _{t}^{T} e^{\beta A_{s}}\overline{Y}_{s-}d\overline{K}_{s}+ 2
\int _{t}^{T} e^{\beta A_{s}}\overline{Y}_{s}d\overline{C}_{s}.
\end{eqnarray}
But for any $\beta >0$ and $\varepsilon > 0 $,
\begin{equation*}
2\overline{Y}_{s} F(s,\overline{\Theta }_{s})\le \left (
\frac{\beta }{2}+ 2 + \frac{2}{\varepsilon }\right )a^{2}_{s} |
\overline{Y}_{s}|^{2} + \frac{2}{\beta }\left |
\frac{\zeta _{s}}{a_{s}}\right |^{2} + \varepsilon (|\overline{Z}_{s}|^{2}+
\|\overline{U}_{s}\|_{\lambda }^{2})
\end{equation*}
and
\begin{eqnarray*}
|g(s,\overline{\Theta }_{s})|^{2} &\leq & 2\left (|g(s,
\overline{\Theta }_{s})-g(s,0)|^{2} + |g(s,0)|^{2}\right )
\nonumber
\\
& \leq & 2\left (a^{2}_{s}|\overline{Y}_{s}|^{2} + \alpha (|
\overline{Z}_{s}|^{2} + \|\overline{U}_{s}\|_{\lambda }^{2})+ |g(s,0)|^{2}
\right ).
\end{eqnarray*}
Plugging these inequalities in \eqref{Ito-Yt} and taking expectation, we
obtain
%
\begin{eqnarray}
\label{eqn8}
&&\left (\frac{\beta }{2} - 4 - \frac{2}{\varepsilon }\right ){\mathbf{E}}
\int _{t}^{T} e^{\beta A_{s}}a^{2}_{s}|\overline{Y}_{s}|^{2}ds + (1 -
\varepsilon - 2\alpha ) {\mathbf{E}}\int _{t}^{T}e^{\beta A_{s}}|
\overline{Z}_{s}|^{2}ds
\nonumber
\\
&&+ (1 - \varepsilon - 2\alpha ){\mathbf{E}}\int _{t}^{T}e^{\beta A_{s}}\|
\overline{U}_{s}\|^{2}_{\lambda }ds
\nonumber
\\
&\leq & {\mathbf{E}}\left ( e^{\beta A_{T}}|\xi |^{2} + \frac{2}{\beta }
\int _{t}^{T}e^{\beta A_{s}}\left |\frac{\zeta _{s}}{a_{s}}\right |^{2}ds
+ 2\int _{t}^{T}e^{\beta A_{s}}|g(s, 0)|^{2}ds\right .
\nonumber
\\
&&\left .+ 2 \int _{t}^{T} e^{\beta A_{s}}\overline{Y}_{s-}d
\overline{K}_{s}+ 2 \int _{t}^{T} e^{\beta A_{s}}\overline{Y}_{s}d
\overline{C}_{s}\right ).
\end{eqnarray}
Moreover,
\begin{eqnarray*}
&&%
\hspace{-1cm}%
2 {\mathbf{E}}\int _{t}^{T} e^{\beta A_{s}}\overline{Y}_{s-}d\overline{K}_{s}+
2 {\mathbf{E}}\int _{t}^{T} e^{\beta A_{s}}\overline{Y}_{s}d\overline{C}_{s}
\\
&&%
\hspace{2cm}%
\;\;\leq \;\;2{\mathbf{E}}\operatorname*{ess
\,sup}_{\tau \in \mathcal{T}_{[0,T]}}e^{\beta A_{\tau }}|\xi _{\tau }|
\int _{0}^{T}d(\overline{K}_{t}+\overline{C}_{t})
\\
&&%
\hspace{2cm}%
\;\;\leq \;\;\varepsilon {\mathbf{E}}\operatorname*{ess
\,sup}_{\tau \in \mathcal{T}_{[0,T]}}e^{2\beta A_{\tau }}|\xi _{\tau }|^{2}+
\frac{1}{\varepsilon }{\mathbf{E}}(\overline{K}_{T}+\overline{C}_{T})^{2},
\end{eqnarray*}
and from \eqref{eqqn} we have
\begin{eqnarray*}
&&{\mathbf{E}}(\overline{K}_{T}+\overline{C}_{T})^{2}
\\
&\leq &6{\mathbf{E}}\left (\overline{Y}_{0}^{2}+\xi _{T}^{2} + \left |
\int _{0}^{T}F(s,\overline{\Theta }_{s})ds\right |^{2} + \left |\int _{0}^{T}g(s,
\overline{\Theta }_{s})dB_{s}\right |^{2}+\left |\int _{0}^{T}
\overline{Z}_{s}dW_{s}\right |^{2} \right .
\\
&&\left .
\hspace{1cm}
+\left |\int _{0}^{T} \int _{E}\overline{U}_{s}(e)\widetilde{\mu }(ds,de)
\right |^{2}\right )
\\
&\leq &6{\mathbf{E}}\left (\overline{Y}_{0}^{2}+\xi _{T}^{2} +
\frac{1}{\beta }\int _{0}^{T}e^{\beta A_{s}}\left |
\frac{F(s,\overline{\Theta }_{s})}{a_{s}}\right |^{2}ds\right .
\\
&&\left .
\hspace{1cm}
+c\left ( \int _{0}^{T}|g(s,\overline{\Theta }_{s})|^{2}ds+\int _{0}^{T}|
\overline{Z}_{s}|^{2}ds +\int _{0}^{T}\|\overline{U}_{s}\|_{\lambda }^{2}
ds\right )\right )
\\
&\leq &6{\mathbf{E}}\left (2\operatorname*{ess
\,sup}_{\tau \in \mathcal{T}_{[0,T]}}e^{2\beta A_{\tau }}|\xi _{\tau }|^{2}+
\frac{4}{\beta }\int _{0}^{T}e^{\beta A_{s}}\left |
\frac{\zeta _{s}}{a_{s}}\right |^{2}ds+2c\int _{0}^{T}e^{\beta A_{s}}|g(s,0)|^{2}ds
\right .
\\
&&\left .
\hspace{1cm}
+\left (\frac{4}{\beta }+2c\right )\int _{0}^{T} e^{\beta A_{s}}a^{2}_{s}|
\overline{Y}_{s}|^{2}ds \right .
\\
&&\left .
\hspace{1cm}
+ \left (\frac{4}{\beta }+2\alpha c+c\right )\left (\int _{0}^{T}e^{
\beta A_{s}}|\overline{Z}_{s}|^{2}ds+\int _{0}^{T}e^{\beta A_{s}}\|
\overline{U}_{s}\|^{2}_{\lambda }ds\right ) \right ).
\end{eqnarray*}
Then \eqref{eqn8} becomes
\begin{eqnarray*}
&&\phi _{1}{\mathbf{E}}\int _{t}^{T} e^{\beta A_{s}}a^{2}_{s}|\overline{Y}_{s}|^{2}ds
+ \phi _{2} {\mathbf{E}}\int _{t}^{T}e^{\beta A_{s}}|\overline{Z}_{s}|^{2}ds+
\phi _{2}{\mathbf{E}}\int _{t}^{T}e^{\beta A_{s}}\|\overline{U}_{s}\|^{2}_{\lambda }ds
\nonumber
\\
&\leq & \Lambda _{1}{\mathbf{E}}\left (\operatorname*{ess
\,sup}_{\tau \in \mathcal{T}_{[0,T]}}e^{2\beta A_{\tau }}|\xi _{\tau }|^{2}
+ \int _{0}^{T}e^{\beta A_{s}}\left |\frac{\zeta _{s}}{a_{s}}\right |^{2}ds
+ \int _{0}^{T}e^{\beta A_{s}}|g(s, 0)|^{2}ds\right ).
\end{eqnarray*}
where
$\phi _{1}=\frac{\beta }{2} - 4 - \frac{2}{\varepsilon }-
\frac{6}{\varepsilon }\bigl (\frac{4}{\beta }+2c \bigr)$,
$\phi _{2}=1 - \varepsilon - 2\alpha -\frac{6}{\varepsilon }\bigl (
\frac{4}{\beta }+2\alpha c+c \bigr)$ and $\Lambda _{1}$ is a nonnegative
constant depending on $\beta $, $c$ and $\varepsilon $. Now,
choose
$ \varepsilon \leq 1- 2\alpha $ with $0<\alpha <1/2$ and $\beta >0$ such
that $\varepsilon \beta (\beta -12-24c)>48$ (these choices are suitable
to obtain a nonnegative $\phi _{1}$ and $\phi _{2}$). Hence
%
\begin{equation}
\label{eqn9}
\left \|  \overline{Y}, \overline{Z}, \overline{U}\right \|  ^{2}_{
\mathcal{A}^{2}_{\beta }( \mathbf{R})} \le \Lambda _{1}\left (\|\xi \|_{
\mathscr{S}^{2}_{2\beta }(\mathbf{R})}^{2} + \left \|
\frac{\zeta }{a}\right \|  ^{2}_{\mathscr{M}^{2}_{\beta }(\mathbf{R})} +
\|g(., 0)\|^{2}_{\mathscr{M}^{2}_{\beta }(\mathbf{R}^{\ell })}\right ).\vadjust{\eject}
\end{equation}
To conclude, we need an estimate of
$\left \|  \overline{Y}\right \|  ^{2}_{{\mathscr{S}}^{2}_{\beta }(
\mathbf{R})}$. For this, using \eqref{Ito-Yt} once again and \eqref{eqn9}, we have
\begin{eqnarray*}
&&{\mathbf{E}}\operatorname*{ess
\,sup}_{\tau \in \mathcal{T}_{[0,T]}}e^{\beta A_{\tau }}|\overline{Y}_{
\tau }|^{2}
\nonumber
\\
&\leq & \Lambda _{1}\left (\|\xi \|_{\mathscr{S}^{2}_{2\beta }(
\mathbf{R})}^{2} + \left \|  \frac{\zeta }{a}\right \|  ^{2}_{
\mathscr{M}^{2}_{\beta }(\mathbf{R})} + \|g(., 0)\|^{2}_{\mathscr{M}^{2}_{\beta }(\mathbf{R}^{\ell })}\right )
\\
&&%
\hspace{-0.5cm}%
+ 2{\mathbf{E}}\operatorname*{ess
\,sup}_{\tau \in \mathcal{T}_{[0,T]}}\left |\int _{0}^{\tau }e^{\beta A_{s}}
\overline{Y}_{s-} g(s, \overline{\Theta }_{s})dB_{s}\right |+2{\mathbf{E}}
\operatorname*{ess
\,sup}_{\tau \in \mathcal{T}_{[0,T]}}\left |\int _{0}^{\tau }e^{\beta A_{s}}
\overline{Y}_{s-}\overline{Z}_{s} dW_{s} \right |
\\
&&%
\hspace{-0.5cm}
+ 2{\mathbf{E}}\operatorname*{ess
\,sup}_{\tau \in \mathcal{T}_{[0,T]}}\left |\int _{0}^{\tau }\int _{E} e^{
\beta A_{s}} \overline{Y}_{s-} \overline{U}_{s}(e)\widetilde{\mu }(ds ,de)
\right |.
\end{eqnarray*}
By the Burkh\"{o}lder--Davis--Gundy inequality, there exists $c > 0$ such that
\begin{equation*}
2{\mathbf{E}}\operatorname*{ess
\,sup}_{\tau \in \mathcal{T}_{[0,T]}}\left |\int _{0}^{\tau }e^{\beta A_{s}}
\overline{Y}_{s-}\overline{Z}_{s} dW_{s} \right |\leq \frac{1}{6}\left
\|  \overline{Y}\right \|  ^{2}_{{\mathscr{S}}^{2}_{\beta }( \mathbf{R})} + 6c^{2}
\left \|  \overline{Z}\right \|  ^{2}_{\mathscr{M}^{2}_{\beta }(\mathbf{R}^{d})},
\end{equation*}
\begin{eqnarray*}
2{\mathbf{E}}\operatorname*{ess
\,sup}_{\tau \in \mathcal{T}_{[0,T]}}\left |\int _{0}^{\tau }e^{\beta A_{s}}
\overline{Y}_{s-} g(s, \overline{\Theta }_{s})dB_{s}\right | &\leq &
\frac{1}{6}\left \|  \overline{Y}\right \|  ^{2}_{{\mathscr{S}}^{2}_{\beta }(
\mathbf{R})} + 6c^{2}\left \|  g(., \overline{\Theta })\right \|  ^{2}_{
\mathscr{M}^{2}_{\beta }(\mathbf{R}^{\ell })}
\\
&&%
\hspace{-3cm}%
\leq \;\;\frac{1}{6}\left \|  \overline{Y}\right \|  ^{2}_{{\mathscr{S}}^{2}_{\beta }( \mathbf{R})} + 12c^{2}\left (\left \|  \overline{Y}\right \|  ^{2}_{{
\mathscr{M}}^{2,a}_{\beta }( \mathbf{R})} + \alpha \left \|  \overline{Z}
\right \|  ^{2}_{\mathscr{M}^{2}_{\beta }(\mathbf{R}^{d})}\right .
\\
&&%
\hspace{-2.5cm}%
\left .+ \alpha \left \|  \overline{U}\right \|  ^{2}_{\mathscr{L}^{2}_{\beta }(\mathbf{R})}+\left \|  g(.,0)\right \|  ^{2}_{\mathscr{M}^{2}_{\beta }(\mathbf{R}^{\ell })}\right )
\end{eqnarray*}
and
\begin{equation*}
2{\mathbf{E}}\operatorname*{ess
\,sup}_{\tau \in \mathcal{T}_{[0,T]}}\left |\int _{0}^{\tau }\int _{E} e^{
\beta A_{s}} \overline{Y}_{s-} \overline{U}_{s}(e)\widetilde{\mu }(ds ,de)
\right |\leq \frac{1}{6}\left \|  \overline{Y}\right \|  ^{2}_{{
\mathscr{S}}^{2}_{\beta }( \mathbf{R})} + 6c^{2}\left \|  \overline{U}
\right \|  ^{2}_{\mathscr{L}^{2}_{\beta }(\mathbf{R})}.
\end{equation*}
Then, we derive that
%
\begin{equation}
\label{eqn13}
\left \|  \overline{Y}\right \|  ^{2}_{{\mathscr{S}}^{2}_{\beta }(
\mathbf{R})} \le \Lambda _{2}\left (\|\xi \|_{\mathscr{S}^{2}_{2
\beta }(\mathbf{R})}^{2} + \left \|  \frac{\zeta }{a}\right \|  ^{2}_{
\mathscr{M}^{2}_{\beta }(\mathbf{R})} + \|g(., 0)\|^{2}_{\mathscr{M}^{2}_{\beta }(\mathbf{R}^{\ell })}\right )
\end{equation}
where $\Lambda _{2}$ is a nonnegative constant depending on $\beta $,
$c$ and $\varepsilon $. The desired result is obtained by combining the
estimates \eqref{eqn9} and \eqref{eqn13} with
$\Lambda =\Lambda _{1}\vee \Lambda _{2}$. As a consequence, from \eqref{estim-Yn} we deduce that
\begin{equation*}
\left \|  Y^{n}\right \|  ^{2}_{{\mathscr{S}}^{2}_{\beta }(\mathbf{R})}
\le \Lambda _{2}\left (\|\xi \|_{\mathscr{S}^{2}_{2\beta }(
\mathbf{R})}^{2} + \left \|  \frac{\zeta }{a}\right \|  ^{2}_{
\mathscr{M}^{2}_{\beta }(\mathbf{R})} + \|g(., 0)\|^{2}_{\mathscr{M}^{2}_{\beta }(\mathbf{R}^{\ell })}\right ).
\end{equation*}
Using the same computations as before, we can prove that
\begin{equation*}
\left \|  Y^{n}, Z^{n}, U^{n}\right \|  ^{2}_{\mathcal{A}^{2}_{\beta }(
\mathbf{R})} \le \Lambda _{1}\left (\|\xi \|_{\mathscr{S}^{2}_{2
\beta }(\mathbf{R})}^{2} + \left \|  \frac{\zeta }{a}\right \|  ^{2}_{
\mathscr{M}^{2}_{\beta }(\mathbf{R})} + \|g(., 0)\|^{2}_{\mathscr{M}^{2}_{\beta }(\mathbf{R}^{\ell })}\right ).\qedhere
\end{equation*}
\end{proof}
%
\begin{lemma}%
\label{lem-converg}
Under the assumption \emph{\textbf{(A2)}} the sequence of processes
$(Y^{n},Z^{n},U^{n})_{n\geq 1}$ converges almost surely in
${\mathscr{B}}^{2}_{\beta }(\mathbf{R})$ for each $\beta>2$.
\end{lemma}
\begin{proof}
We know that
%
\begin{eqnarray}
\label{Y-n}
Y^{n}_{t} &=&\xi _{T} + \int _{t}^{T}f_{n}(s,\Theta ^{n}_{s})ds +
\int _{t}^{T}g(s,\Theta ^{n}_{s})dB_{s}-\int _{t}^{T}Z^{n}_{s}dW_{s}
\nonumber
\\
&&- \int _{t}^{T} \int _{E}U^{n}_{s}(e)\widetilde{\mu }(ds,de)+K^{n}_{T}-K^{n}_{t}+C^{n}_{T-}-C^{n}_{t-},
\end{eqnarray}
where $(K^{n},C^{n})$ satisfy the Skorokhod\index{Skorokhod condition} and minimality conditions.\index{minimality condition}
We define, for any integers $n, m \geq 1$,
$\Re ^{n,m} = \Re ^{n}-\Re ^{m}$ for $\Re \in \{Y, Z, U, K,C\}$,
\begin{equation*}
\Delta f^{n, m}(t) = f_{n}(t, \Theta ^{n}_{t}) - f_{m}(t,\Theta ^{m}_{t})
\;\;\text{and}\;\; \Delta g^{n, m}(t) = g(t, \Theta ^{n}_{t}) - g(t,
\Theta ^{m}_{t}),\quad t \leq T.
\end{equation*}
Then, applying Lemma \ref{lem1} together with Remark \ref{rem02}, we
get
\begin{eqnarray*}
&&%
\hspace{-1cm}%
\beta {\mathbf{E}}\int _{t}^{T}e^{\beta A_{s}}a^{2}_{s}|Y^{n, m}_{s}|^{2}ds
+ {\mathbf{E}}\int _{t}^{T}e^{\beta A_{s}}| Z^{n, m}_{s}|^{2}ds+ {\mathbf{E}}
\int _{t}^{T}e^{\beta A_{s}}\|U^{n, m}_{s}\|^{2}_{\lambda }ds
\\
&\leq &{\mathbf{E}}\int _{t}^{T} e^{\beta A_{s}}|\Delta g^{n, m}(s)|^{2}ds+2{
\mathbf{E}}\int _{t}^{T} e^{\beta A_{s}} Y^{n, m}_{s} \Delta f^{n, m}(s)ds.
\end{eqnarray*}
Using the assumption \textbf{(A2.3)} and the basic inequality
$2ab\leq \epsilon a^{2}+\frac{b^{2}}{\epsilon }$, we get
\begin{eqnarray*}%
&&(\beta - 1)\left \|  Y^{n, m}\right \|  ^{2}_{\mathscr{M}^{2,a}_{\beta }(
\mathbf{R})} + (1-\alpha )\left (\left \|  Z^{n, m}\right \|  ^{2}_{
\mathscr{M}^{2}_{\beta }(\mathbf{R}^{d})} + \left \|  U^{n, m}\right \|  ^{2}_{
\mathscr{L}^{2}_{\beta }(\mathbf{R})}\right )
\\
&\le &{\mathbf{E}}\int _{0}^{T} e^{\beta A_{s}} a_{s}^{2}|Y^{n, m}_{s}|^{2}
ds + {\mathbf{E}}\int _{0}^{T}e^{\beta A_{s}}\left |
\frac{\Delta f^{n, m}(s)}{a_{s}}\right |^{2}ds.
\end{eqnarray*}
Next, from the linear growth condition on $f_{n}$ and $f_{m}$, and by Lemma \ref{borne}, we find
\begin{eqnarray*}%
&&(\beta - 2)\left \|  Y^{n, m}\right \|  ^{2}_{\mathscr{M}^{2,a}_{\beta }(
\mathbf{R})} + (1-\alpha )\left (\left \|  Z^{n, m}\right \|  ^{2}_{
\mathscr{M}^{2}_{\beta }(\mathbf{R}^{d})} + \left \|  U^{n, m}\right \|  ^{2}_{
\mathscr{L}^{2}_{\beta }(\mathbf{R})}\right )
\\
&\le &8\Lambda \left (\|\xi \|_{\mathscr{S}^{2}_{2\beta }(\mathbf{R})}^{2}
+ \left \|  \frac{\zeta }{a}\right \|  ^{2}_{\mathscr{M}^{2}_{\beta }(
\mathbf{R})} + \|g(., 0)\|^{2}_{\mathscr{M}^{2}_{\beta }(\mathbf{R}^{
\ell })}\right ).
\end{eqnarray*}
Hence for $\beta > 2$, we deduce that $(Y^{n}, Z^{n},U^{n})$ is a Cauchy
sequence in $\mathcal{A}^{2}_{\beta }(\mathbf{R})$, so it converges in
$\mathcal{A}^{2}_{\beta }(\mathbf{R})$. On the other hand, from \eqref{estim-Yn} we deduce that there exists a process
$Y\in \mathscr{S}^{2}_{\beta }( \mathbf{R})$ such that
$Y^{n} \rightarrow Y$ a.s. as $n \rightarrow \infty $. The result follows.
\end{proof}

The main result in this section is what follows. 
%
\begin{theorem}
Under the assumptions \emph{\textbf{(A2)}}, the RBDSDEJ \eqref{backw} associated
with parameters $(f(.,\Theta ),g(.,\Theta ),\xi )$ has a minimal solution
$(Y,Z,U,K,C)\in \mathscr{B}^{2}_{\beta }(\mathbf{R})\times \mathscr{S}^{2}(
\mathbf{R})\times \mathscr{S}^{2}(\mathbf{R})$.\index{minimal solution}
\end{theorem}
\begin{proof}
From \eqref{estim-Yn}, it is readily seen that $(Y^{n})_{n\geq 1}$ converges
to $Y$ a.s. in $\mathscr{S}^{2}_{\beta }(\mathbf{R})$. Otherwise, due
to Lemma \ref{lem-converg}\vadjust{\goodbreak} there exist two subsequences still noted as
the whole sequences $(Z^{n})_{n\geq 1}$ and $(U^{n})_{n \geq 1}$ such that
$\Theta ^{n}=(Y^{n}, Z^{n},U^{n})$ converges to
$\Theta =(Y,Z,U) \in \mathcal{A}^{2}_{\beta }(\mathbf{R})$ as
$n \rightarrow +\infty $. By \eqref{convergence}, we have
\begin{equation*}
f_{n}(t,\Theta _{t}^{n})\xrightarrow[n \to +\infty ]{} f(t,\Theta _{t}),
\quad t \le T.
\end{equation*}
Furthermore, using the linear growth condition of $f_{n}$, it follows that
\begin{eqnarray*}
&&{\mathbf{E}}\int _{0}^{T}e^{\beta A_{s}}\left |
\frac{f_{n}(s, \Theta _{s}^{n})}{a_{s}}\right |^{2}ds
\\
&\le &4{\mathbf{E}}\bigg (\int _{0}^{T}e^{\beta A_{s}}\left |
\frac{\zeta _{s}}{a_{s}}\right |^{2}ds + \sup _{n}\int _{0}^{T}e^{
\beta A_{s}}a^{2}_{s}|Y^{n}_{s}|^{2}ds + \sup _{n}\int _{0}^{T}e^{
\beta A_{s}}|Z_{s}^{n}|^{2}ds
\\
&&+ \sup _{n}\int _{0}^{T}e^{\beta A_{s}}\|U_{s}^{n}\|^{2}_{\lambda }ds
\bigg ),
\end{eqnarray*}
and by Lemma \ref{borne} we deduce that
\begin{eqnarray*}
&&{\mathbf{E}}\int _{0}^{T}e^{\beta A_{s}}\left |
\frac{f_{n}(s,\Theta _{s}^{n})}{a_{s}}\right |^{2}ds
\\
&\le & 4\left (\Lambda \|\xi \|_{\mathscr{S}^{2}_{2\beta }(
\mathbf{R})}^{2} + (1+\Lambda ) \left \|  \frac{\zeta }{a}\right \|  ^{2}_{
\mathscr{M}^{2}_{\beta }(\mathbf{R})} + \Lambda \|g(., 0)\|^{2}_{
\mathscr{M}^{2}_{\beta }(\mathbf{R}^{\ell })}\right ).
\end{eqnarray*}
Since
\begin{equation*}
{\mathbf{E}}\left |\int _{0}^{T}f_{n}(s,\Theta _{s}^{n})ds\right |^{2}
\le \frac{1}{\beta }{\mathbf{E}}\int _{0}^{T}e^{\beta A_{s}}\left |
\frac{f_{n}(s,\Theta _{s}^{n})}{a_{s}}\right |^{2}ds,
\end{equation*}
by Lebesgue's dominated convergence theorem, we deduce that, for almost
all $t\leq T$,
\begin{equation*}
{\mathbf{E}}\left |\int _{0}^{T}(f_{n}(s,\Theta _{s}^{n})-f(s,\Theta _{s}))ds
\right |^{2}\xrightarrow[n \to +\infty ]{} 0.
\end{equation*}
We have also, for almost all $t\leq T$,
\begin{equation*}
{\mathbf{E}}\left |\int _{0}^{T}(g(s,\Theta _{s}^{n})-g(s,\Theta _{s}))dB_{s}
\right |^{2}\xrightarrow[n \to +\infty ]{} 0.
\end{equation*}
Moreover, we have
\begin{equation*}
{\mathbf{E}}\left (\sup _{0\leq t \leq T}\left |\int _{t}^{T} Z_{s}^{n}dW_{s}
- \int _{t}^{T} Z_{s}dW_{s}\right |^{2}\right )\leq {\mathbf{E}}\int _{t}^{T}e^{
\beta A_{s}}|Z_{s}^{n}-Z_{s}|^{2}ds\xrightarrow[n \to +\infty ]{} 0
\end{equation*}
and
\begin{align*}
{\mathbf{E}}\bigg (\sup _{0\leq t \leq T}\bigg |\int _{t}^{T}\int _{E} U_{s}^{n}(e)
\widetilde{\mu }(de,ds) &- \int _{t}^{T}\int _{E} U_{s}(e)
\widetilde{\mu }(de,ds)\bigg |^{2}\bigg )
\\
&\leq \;\; {\mathbf{E}}\int _{t}^{T}e^{\beta A_{s}}\|U_{s}^{n}-U_{s}\|^{2}_{\lambda }ds\xrightarrow[n \to +\infty ]{} 0.
\end{align*}
Next, for each $\tau \in \mathcal{T}_{[0,T]}$, let
\begin{eqnarray*}
\widetilde{K}_{\tau }=K_{\tau }-C_{\tau -}&=&Y_{0}-Y_{\tau }-\displaystyle
\int _{0}^{\tau }f\left (s,\Theta _{s}\right )ds-\int _{0}^{\tau }g\left (s,
\Theta _{s}\right )dB_{s}
\\
&&+\int _{0}^{\tau }Z_{s} dW_{s}+\displaystyle \int _{0}^{\tau }\int _{E} U_{s}(e)
\widetilde{\mu }(ds,de).
\end{eqnarray*}
Then, we can easy show that
$\|\widetilde{K}^{n}-\widetilde{K}\|^{2}_{\mathscr{S}^{2}}
\longrightarrow 0,\,\,\text{as}\,\,n\longrightarrow +\infty $. So, letting
${n \longrightarrow + \infty }$ in \eqref{Y-n}, we deduce that
$(Y,Z,U,K,C)$ is a solution to RBDSDEJ \eqref{backw}.

Now, let
$(Y^{\ast },Z^{\ast },U^{\ast },K^{\ast },C^{\ast })\in \mathscr{B}^{2}_{\beta }(
\mathbf{R})\times \mathscr{S}^{2}(\mathbf{R})\times \mathscr{S}^{2}(
\mathbf{R})$ be another solution to RBDSDEJ \eqref{backw}. By virtue of
Theorem \ref{comp-Lipsch}, we deduce that
\begin{equation*}
\forall n \geq 1,\quad Y^{n} \le Y^{\ast }.
\end{equation*}
Therefore, by passing to the limit $Y\leq Y^{\ast }$ 
one proves that $Y$ is the
minimal solution to RBDSDEJ \eqref{backw}.
\end{proof}



\begin{acknowledgement}[title=Acknowledgments]
The authors are greatly grateful to the editor and the referee for the
careful reading and many constructive suggestions, which significantly
contributed to improving the quality of the paper.
\end{acknowledgement}

\begin{funding}
The corresponding author (Mohamed Marzougue) declares that this research was supported by National Center
for Scientific and Technical Research (\gsponsor[id=GS1,sponsor-id=501100006319]{CNRST}), Morocco.
\end{funding}


\end{document}